\documentclass{amsart}

\usepackage{amsmath, amsthm, amsfonts, amssymb, graphicx, tikz-cd, comment, array, hyperref, extarrows, mathtools, todonotes, mleftright}

\usepackage{enumitem}
\usepackage{stmaryrd} 
\usepackage[capitalise]{cleveref} 

\newtheorem{theorem}{Theorem}[section]
\newtheorem{lemma}[theorem]{Lemma}
\newtheorem{proposition}[theorem]{Proposition}
\newtheorem{corollary}[theorem]{Corollary}
\newtheorem{claim}[theorem]{Claim}

\theoremstyle{definition}
\newtheorem{definition}[theorem]{Definition}
\newtheorem{notation}[theorem]{Notation}

\newtheorem{remark}[theorem]{Remark}

\def\KHaus{{\mathsf{KHaus}}}
\def\Stone{{\mathsf{Stone}}}
\def\Clop{{\mathsf{Clop}}}
\def\RO{{\mathsf{RO}}}
\def\RC{{\mathsf{RC}}}
\def\ba{{\mathsf{BA}}}

\def\G{{\mathcal{G}}}
\def\int{{\mathsf{int}}}
\def\cl{{\mathsf{cl}}}

\def\dev{{\mathsf{DeV}}}

\def\R{{\mathfrak{R}}}
\def\RO{{\mathsf{RO}}}

\def\Ult{{\mathsf{Ult}}}
\def\rofunc{\mathcal{D}}
\def\Qfunc{{\mathcal{Q}}}

\def\R{\mathrel{R}}
\def\S{\mathrel{S}}
\def\T{\mathrel{T}}
\def\Q{\mathrel{Q}}
\def\E{\mathrel{E}}

\newcommand{\StoneR}{\mathsf{Stone}^{\mathsf{R}}}

\newcommand{\KHausR}{\mathsf{KHaus}^\mathsf{R}}
\newcommand{\BAS}{\mathsf{BA}^\mathsf{S}}

\newcommand{\BAQ}{\mathsf{BA}^\mathsf{Q}}

\newcommand{\Gle}{\mathsf{Gle}}
\newcommand{\GleR}{\mathsf{Gle}^\mathsf{R}}

\newcommand{\StoneER}{\mathsf{StoneE}^\mathsf{R}}
\newcommand{\baS}{\ba^\mathsf{S}}
\newcommand{\SubSfive}{\mathsf{SubS5^S}}

\newcommand{\devS}{\mathsf{DeV^S}}
\newcommand{\devF}{\mathsf{DeV^F}}

\newcommand{\id}{\mathrm{id}}

\newcommand{\StoneEf}{\mathsf{StoneE^F}}
\newcommand{\SubSfivef}{\mathsf{SubS5^F}}

\newcommand{\Rel}{\mathsf{Rel}}
\newcommand{\Map}{\mathsf{Map}}

\newcommand\conv[1]{{#1 \raisebox{-2pt}{\scalebox{1.2}{$\breve{\hspace{3.5pt}}$}}}}

\newcommand\convrs[1]{{\conv{(#1)}}}
\newcommand\convgxprime[1]{{\conv{(#1)}}}

\newcommand\inv[1]{\widehat{#1}}

\newcommand{\C}{{\mathsf{C}}}
\newcommand{\A}{{\mathsf{A}}}
\newcommand{\Eq}{\mathsf{Eq}}
\newcommand{\Split}{\mathsf{Split}}

\title[A generalization of de Vries duality to closed relations]{A generalization of de Vries duality to closed relations between compact Hausdorff spaces}

\author[M. Abbadini]{Marco Abbadini}
\address{Department of Mathematics\\
University of Salerno\\
84084 Fisciano (SA)\\
Italy}
\email{mabbadini@unisa.it}

\author[G. Bezhanishvili]{Guram Bezhanishvili}
\address{Department of Mathematical Sciences\\
New Mexico State University\\
Las Cruces NM 88003\\
USA}
\email{guram@nmsu.edu}

\author[L. Carai]{Luca Carai}
\address{Department of Philosophy\\
University of Barcelona\\
08001 Barcelona\\
Spain}
\email{luca.carai.uni@gmail.com}

\date{}


\begin{document}

\maketitle

\begin{abstract}
	Stone duality 
	generalizes to an equivalence 
	between the categories $\StoneR$ of Stone spaces and closed relations and 
	$\BAS$ of boolean algebras and subordination relations. 
Splitting equivalences in $\StoneR$
	yields a category that is equivalent to 
	the category $\KHausR$ of compact Hausdorff spaces and closed relations. 
	Similarly, splitting equivalences in $\BAS$ yields a category that 
	is equivalent to the category $\devS$ of de Vries algebras and compatible subordination relations. Applying the machinery of allegories then yields that $\KHausR$ is equivalent to $\devS$, thus resolving a problem recently raised in the literature.
	
	The equivalence between $\KHausR$ and $\devS$ further restricts to an equivalence between the category $\KHaus$ of compact Hausdorff spaces and continuous functions and the wide subcategory $\devF$ of $\devS$ whose morphisms satisfy additional conditions.
	This yields an alternative to de Vries duality.
	One advantage of this approach is that composition of morphisms is usual relation composition.
\end{abstract}

\section{Introduction}

In \cite{deV62} de Vries generalized Stone duality to a duality for the category $\KHaus$ of compact Hausdorff spaces and continuous maps. 
The objects of the dual category $\dev$ are complete boolean algebras equipped with a proximity relation, known as de Vries algebras. 
The morphisms of $\dev$ are functions satisfying certain conditions. 
One drawback of $\dev$ is that composition of morphisms is not usual function composition.
In this paper we propose an alternative approach to de Vries duality, where morphisms between de Vries algebras become certain relations and composition is usual relation composition. 

For our purpose, it is more natural to start with the category $\KHausR$ whose objects are compact Hausdorff spaces and morphisms are closed relations (i.e., relations $R \colon X\to Y$ such that $R$ is a closed subset of $X \times Y$). This category was studied in \cite{BezhanishviliGabelaiaEtAl2019}, and earlier in \cite{JungKegelmannEtAl2001} in the more general setting of stably compact spaces. The latter paper establishes a duality for $\KHausR$ that generalizes Isbell duality \cite{Isb72} between $\KHaus$ and the category of compact regular frames and frame homomorphisms. This is obtained by generalizing the notion of a frame homomorphism to that of a preframe homomorphism. 
However, a similar duality in the language of de Vries algebras remained problematic (see \cite[Rem.~3.14]{BezhanishviliGabelaiaEtAl2019}). We resolve this problem as follows. 

By Stone duality, the category $\Stone$ of Stone spaces (zero-dimensional compact Hausdorff spaces) and continuous maps is dually equivalent to the category $\ba$ of boolean algebras and boolean homomorphisms. Halmos \cite{Hal56} generalized Stone duality by showing that continuous relations between Stone spaces dually correspond to functions $f\colon A\to B$ between boolean algebras that preserve finite meets. 
This was further generalized by Celani \cite{Cel18} to closed relations between Stone spaces and certain functions from a boolean algebra $A$ to the ideal frame $\mathcal J(B)$ of a boolean algebra $B$. We show that such functions correspond to subordination relations between $A$ and $B$ studied in \cite{BBSV17}. Consequently, we obtain that the category $\StoneR$ of Stone spaces and closed relations is equivalent to the category $\baS$ of boolean algebras and subordination relations (identity is $\le$ and composition is usual composition of relations). This equivalence is in fact an equivalence of allegories, hence self-dual categories. Thus, the choice of direction of morphisms is ultimately a matter of taste and the equivalence can alternatively be stated as a dual equivalence.

We point out that the equivalence between $\StoneR$ and $\baS$ is also a consequence of a more general result of Jung, Kurz, and Moshier \cite[Thm.~5.9]{KurzMoshierEtAl2019} who worked with order-enriched categories to show that Priestley duality extends to a dual equivalence between the category of bounded distributive lattices and subordination relations and the category of Priestley spaces and Priestley relations. 

As we pointed out above, both $\StoneR$ and $\baS$ are allegories, and the equivalence between $\StoneR$ and $\baS$ is an equivalence of allegories. Therefore, we may utilize the machinery of allegories \cite{FS90} to split equivalences in both $\StoneR$ and $\baS$. Splitting equivalences in $\StoneR$ yields the category $\StoneER$ which is equivalent to $\KHausR$. On the other hand, splitting equivalences in $\baS$ yields the category $\SubSfive$ which we show is equivalent to $\devS$. By piggybacking the equivalence of $\StoneR$ and $\baS$, we obtain an equivalence between $\StoneER$ and $\SubSfive$, which then yields our desired equivalence between $\KHausR$ and $\devS$. This resolves an open problem raised in \cite[Rem.~3.14]{BezhanishviliGabelaiaEtAl2019} (see \cref{rem:open problem}).

One drawback of $\StoneER$ is that isomorphisms are not structure-preserving bijections. However, $\StoneER$ has a full subcategory $\GleR$ of Gleason spaces \cite{BBSV17,BezhanishviliGabelaiaEtAl2019} which is more directly related to $\devS$. We prove that in both $\devS$ and $\GleR$ isomorphisms are structure-preserving bijections. 

The equivalence between $\KHausR$ and $\devS$ can also be obtained by directly generalizing the regular open functor of de Vries duality. Indeed,
associate with each compact Hausdorff space $X$ the de Vries algebra $(\RO(X),\prec)$, where $\RO(X)$ is the complete boolean algebra of regular open subsets of $X$ and $\prec$ is the de Vries proximity on $\RO(X)$ defined by $U \prec V$ iff $\cl(U) \subseteq V$.
The key is to associate with each closed relation $R \colon X \to Y$ the relation $S_R \colon \RO(X) \to \RO(Y)$ given by 
\[
U \S_R V \iff R[\cl(U)]\subseteq V
\] 
(here $R[-]$ denotes the direct image under $R$).
This gives a more explicit description of the equivalence 
between $\KHausR$ and $\devS$, which further restricts to an equivalence between $\KHaus$ and a wide subcategory $\devF$ of $\devS$, thus providing an alternative to de Vries duality. 

This paper is related to the line of research initiated by D.~Scott~\cite{Scott1982}, and further developed in \cite{LarsenWinskel1984,Hoofman1993,Vickers1993,AbramskyJung1994,JungSuenderhauf1996,JungKegelmannEtAl1999,JungKegelmannEtAl2001,Kegelmann2002,Moshier2004,KurzMoshierEtAl2019}, that uses certain relations as morphisms.
We apply the insights of these works specifically to the category of compact Hausdorff spaces and closed relations. We do so in a way that is optimal to explain the connection with de Vries duality.
The setting of compact Hausdorff spaces allows the use of simple entities such as closed equivalence relations on Stone spaces and their quotients, which are familiar to a wide range of mathematicians approaching topology from an algebraic or logical perspective.

\section{Lifting Stone duality to closed relations}

For two sets $X$ and $Y$, we write $R \colon X \to Y$ to indicate that $R$ is a relation from $X$ to $Y$. As usual, for $F\subseteq X$, we write $R[F]$ for the $R$-image of $F$ in $Y$; and for $G\subseteq Y$, we write $R^{-1}[G]$ for the $R$-inverse image of $G$ in $X$. 

If $X,Y$ are Stone spaces, then we call $R \colon X \to Y$ {\em closed} if $R$ is a closed subset of $X\times Y$ (equivalently, $R[F]$ is closed for each closed $F\subseteq X$ and $R^{-1}[G]$ is closed for each closed $G\subseteq Y$).

\begin{definition}
Let $\StoneR$ be the category of Stone spaces and closed relations between them. Identity morphisms in $\StoneR$ are identity relations and composition in $\StoneR$ is relation composition. 
\end{definition}
 
 As we pointed out in the introduction, Celani \cite{Cel18} extended Stone duality to $\StoneR$ by generalizing boolean homomorphisms to what he termed ``quasi-semi-homomorphisms." For a boolean algebra $B$, let $\mathcal{J}(B)$ be the complete lattice of ideals of $B$. 

\begin{definition} \cite[Def.~1]{Cel18}
Let $A,B$ be boolean algebras. A {\em quasi-semi-homomorphism} is a function $\Delta \colon A\to\mathcal{J}(B)$ such that $\Delta(1)=B$ and $\Delta(a\wedge b)=\Delta(a)\cap\Delta(b)$ for all $a,b\in A$.
\end{definition}

By \cite[Lem.~2]{Cel18}, boolean algebras and quasi-semi-homomorphisms between them form a category, which we denote by $\BAQ$. The identity quasi-semi-homomorphism on $A$ is given by $I_A(a)={\downarrow}a$ and the composition of two quasi-semi-homomorphisms $\Delta_1 \colon A \to B$ and $\Delta_2 \colon B \to C$ by
\[
(\Delta_2 \circ \Delta_1)(a) = \bigcup \{ \Delta_2(b) \mid b \in \Delta_1(a) \}
\]
(note that this union is an ideal because $\{ \Delta_2(b) \mid b \in \Delta_1(a) \}$ is directed).

\begin{theorem} [{\cite[Thm.~4]{Cel18}}] \label{thm: Celani}
$\StoneR$ is dually equivalent to $\BAQ$.
\end{theorem}

Quasi-semi-homomorphisms from A to $\mathcal J(A)$ were studied in \cite{Celani2001} under the name of quasi-modal operators. It was pointed out in \cite[Rem.~2.6]{BBSV17} that quasi-modal operators on $A$ are in one-to-one correspondence with subordinations on $A$. We show that this generalizes to a dual isomorphism between $\BAQ$ and the category of boolean algebras and subordination relations between them. 

\begin{definition} \label{def:subordination}
	\begin{enumerate}
		\item[]
		\item A \emph{subordination relation} from a boolean algebra $A$ to a boolean algebra $B$ is a relation $S \colon A \to B$ satisfying, for $a,b\in A$ and $c,d\in B$:
		\begin{enumerate}[label = (S\arabic*), ref = S\arabic*]
			\item \label{S1} $0 \S 0$ and $1 \S 1$;
			\item \label{S2} $a,b \S c$ implies $(a\vee b) \S c$; 
			\item \label{S3} $a \S c,d$ implies $a \S (c\wedge d)$;  
			\item \label{S4} $a\le b \S c\le d$ implies $a \S d$. 
		\end{enumerate}
		\item Let $\BAS$ be the category of boolean algebras and subordination relations. The identity morphism on $A$ is the order relation $\le$ on $A$, and composition in $\BAS$ is relation composition. 
	\end{enumerate}
\end{definition}

\begin{theorem} \label{thm: BAS=BAQ}
$\BAS$ is dually isomorphic to $\BAQ$.
\end{theorem}

\begin{proof}
For a subordination $S \colon A\to B$, define $\Delta_S \colon B \to \mathcal J(A)$ by $\Delta_S(b) = S^{-1}[b]$. It is straightforward to check that $\Delta_S$ is a well-defined quasi-semi-homomorphism. Moreover, $\Delta_{\le_A}(a) = {\downarrow}a$, so $\Delta_{\le_A} = I_A$. Furthermore, if $S_1 \colon A \to B$ and $S_2 \colon B \to C$ are subordinations, then 
\[
\Delta_{S_2 \circ S_1}(c) = (S_2 \circ S_1)^{-1}[c] = S_1^{-1}S_2^{-1}[c] 
= (\Delta_{S_1} \circ \Delta_{S_2})(c). 
\]

For a quasi-semi-homomorphism $\Delta \colon A \to B$, define $S_\Delta \colon B \to A$ by $b \S_\Delta a$ iff $b \in \Delta(a)$. Again, it is straightforward to check that $S_\Delta$ is a well-defined subordination. Moreover, $b \S_{I_A} a$ iff $b \in {\downarrow} a$ iff $b\le a$, so $S_{I_A} = {\le_A}$. Furthermore, if $\Delta_1 \colon A \to B$ and $\Delta_2 \colon B \to C$ are quasi-semi-homomorphisms, then
\begin{eqnarray*}
c \S_{\Delta_2 \circ \Delta_1} a & \Longleftrightarrow & c \in (\Delta_2 \circ \Delta_1)(a) \Longleftrightarrow c \in \bigcup \{ \Delta_2(b) \mid b \in \Delta_1(a) \} \\
& \Longleftrightarrow & \exists b\in\Delta_1(a) : c \in \Delta_2(b) \Longleftrightarrow c \ (\S_{\Delta_1} \circ \S_{\Delta_2}) \ a.
\end{eqnarray*}

In addition, for each subordination $S$, we have
$
b \S_{\Delta_S} a \Longleftrightarrow b \in \Delta_S(a) \Longleftrightarrow b \S a,
$ 
so $S_{\Delta_S} = S$. Also, for each quasi-semi-homomorphism $\Delta$, we have
$
\Delta_{S_\Delta}(a) = S_\Delta^{-1}[a] = \Delta(a),
$
and hence $\Delta_{S_\Delta}=\Delta$. 
Thus, $\BAS$ is dually isomorphic to $\BAQ$.
\end{proof}

Putting \cref{thm: Celani,thm: BAS=BAQ} together yields:

\begin{corollary} \label{cor: StoneR=BAS}
$\StoneR$ is equivalent to $\BAS$.
\end{corollary}

\begin{remark} \label{rem: Clop and Uf}
The functors establishing the equivalence of \cref{cor: StoneR=BAS} generalize the well-known clopen and ultrafilter functors $\Clop$ and $\Ult$ of Stone duality. Indeed, if
$R \colon X \to Y$ is a closed relation between Stone spaces, then composing the functors $\StoneR\to\BAQ\to\BAS$, we obtain that the corresponding subordination relation $S_R \colon \Clop(X) \to \Clop(Y)$ is given by 
	$U \S_R V$ iff 
	$R[U] \subseteq V$. 
	Conversely, if 
		$S \colon A \to B$ is a subordination relation, then composing the functors $\BAS\to\BAQ\to\StoneR$ in the other direction yields that the corresponding closed relation $R_S \colon \Ult(A) \to \Ult(B)$ is given by
	$
	x \R_S y \mbox{ iff } S[x] \subseteq y.
	$
We will slightly abuse the notation and use $\Clop$ and $\Ult$ also for the functors establishing the equivalence between $\StoneR$ and $\BAS$.
\end{remark}

\begin{remark}
That in \cref{cor: StoneR=BAS} we have an equivalence instead of a  dual equivalence
is explained by the fact that the categories $\StoneR$ and $\BAS$ are self-dual (see \cref{lem:StoneR and BAS allegories}), and hence equivalence implies dual equivalence.
\end{remark}

\begin{remark}
\cref{cor: StoneR=BAS} is also a consequence of \cite[Thm.~5.9]{KurzMoshierEtAl2019}, where a more general order-enriched duality result is established
between distributive lattices and Priestley spaces, with appropriate relations as morphisms.
\end{remark}

From now on we will work with the additional structure of allegory on $\StoneR$ and $\BAS$ \cite{FS90,Joh02}. 
In fact, it is enough to have the structure of order-enriched category with involution \cite{TsalenkoGisinEtAl1984,Lambek1999}. However, we prefer to work with allegories because their well-developed machinery is readily available to us. One could instead work with bicategories (see \cite{CW87,CV98}).

\begin{definition}\label{def:allegory}
\mbox{}\begin{enumerate}
\item{\cite[p.~105]{HST14}}\label{def:allegory:item locally posetal} A category $\C$ is \emph{order enriched} if each hom-set of $\C$ is equipped with a partial order $\le$ such that $f \le f'$ and $g \le g'$ imply $gf \le g'f'$ for all $f,f' \colon C \to D$ and $g,g' \colon D \to E$.
\item{\cite[p.~74]{HV19}} A \emph{dagger} on a category $\C$ is a contravariant endofunctor $(-)^\dagger \colon \C \to \C$ that is the identity on objects and the composition of $(-)^\dagger$ with itself is the identity on $\C$. A \emph{dagger category} is a category equipped with a dagger.
\item{\cite[p.~136]{Joh02}}\label{def:allegory:item allegory} An \emph{allegory} $\A$ is an order-enriched dagger category such that: 
\begin{enumerate}[label = (\roman*), ref = (\roman*)]
\item each partially ordered hom-set of $\A$ has binary meets;
\item $(-)^\dagger$ preserves the order on the hom-sets ($f \le g \ \Longrightarrow \ f^\dagger \le g^\dagger$);
\item the modular law holds: $g f \wedge h \le (g \wedge h f^\dagger) f$  for all $f \colon C \to D$, $g \colon D \to E$, and $h \colon C \to E$.
\end{enumerate}
\end{enumerate}
\end{definition}

We next extend the notion of equivalence of categories to allegories.

\begin{definition}\label{def:morphism of allegories}
\mbox{}\begin{enumerate}
\item{\cite[p.~142]{Joh02}}\label{def:allegory:item allegory morphism} If $F \colon \A \to \A'$ is a functor between two allegories, we say that $F$ is a \emph{morphism of allegories} if it preserves the binary meets of morphisms and commutes with $(-)^\dagger$.

\item\label{def:allegory:item allegory equivalence} A morphism of allegories is an \emph{equivalence of allegories} if it is an equivalence of categories.
\end{enumerate}
\end{definition}

\begin{remark}\label{rem:quasi-inverse allegories}
If $F \colon \mathsf{A} \to \mathsf{A'}$ is an equivalence of allegories with quasi-inverse $G \colon \mathsf{A'} \to \mathsf{A}$, then $G$ is also an equivalence of allegories.
To see this, if $f$ is an isomorphism in $\mathsf{A'}$, then $f^{-1}=f^\dagger$ (see~\cite[p.~199]{FS90}).
Therefore, since $F$ is an essentially surjective, full, and faithful functor that preserves the dagger, the proof of \cite[Lem.~V.1]{Vic11} shows that $G$ also preserves the dagger. Because $F$ is a bijection on the hom-sets that preserves meets, it is an order-isomorphism on the hom-sets. Thus, $G$ is also an order-isomorphism on the hom-sets, and hence an equivalence of allegories.
\end{remark}

\begin{lemma}\label{lem:StoneR allegory}
$\StoneR$ is an allegory.
\end{lemma}

\begin{proof}
$\StoneR$ is an order-enriched category if we order its hom-sets by inclusion.
For a relation $R \colon X \to Y$, let $\conv{R} \colon Y \to X$ be its converse defined by 
$y \mathrel{\conv{R}} x$ iff $x \R y$.
If $R$ is a closed relation, then $\conv{R}$ is also closed. The assignment $R \mapsto \conv{R}$ defines a dagger on $\StoneR$.
It is straightforward to check that all the conditions of \cref{def:allegory}\eqref{def:allegory:item allegory} hold in $\StoneR$. 
\end{proof}

\begin{theorem}\label{lem:StoneR and BAS allegories}
$\BAS$ is an allegory and $\Clop$ and $\Ult$ yield an equivalence of $\StoneR$ and $\BAS$ as allegories.
\end{theorem}

\begin{proof}
If we order the hom-sets of $\BAS$ by reverse inclusion, then $\BAS$ becomes an order-enriched category. 
For a subordination $S \colon A \to B$, define $\inv{S} \colon B \to A$ by $b \mathrel{\inv{S}} a$ iff $\neg a \S \neg b$. It is straightforward to check that $\inv{S}$ is a subordination and that the assignment $S \mapsto \inv{S}$ defines a dagger on $\BAS$.
We show that the bijections between the hom-sets induced by the functors $\Clop$ and $\Ult$ preserve and reflect the orders, and commute with the daggers.

Suppose $R_1 \subseteq R_2$. Let $U \in \Clop(X)$ and $V \in \Clop(Y)$ such that $U \S_{R_2} V$. Then $R_1[U] \subseteq R_2[U] \subseteq V$, and so $U \S_{R_1} V$. Thus, $S_{R_2} \subseteq S_{R_1}$. Conversely, assume that $R_1 \nsubseteq R_2$. Let $x \in X$ and $y \in Y$ be such that $x \R_1 y$ but $x \not\R_2 y$. Since $R_2$ is a closed relation, $R_2[x]$ is a closed subset of $Y$ and $y \notin R_2[x]$. Because $Y$ is a Stone space, there is clopen $V \subseteq Y$ such that $R_2[x] \subseteq V$ and $y \notin V$. Since $R_2$ is closed, $X \setminus R_2^{-1}[Y \setminus V]$ is an open subset of $X$ containing $x$. Therefore, there is $U \subseteq X$ clopen such that $x \in U$ and $X \setminus U \subseteq R_2^{-1}[Y \setminus V]$, so $R_2[U] \subseteq V$. Thus, $x \in U$, $y \notin V$, and $x \R_1 y$, yielding $R_1[U] \nsubseteq V$. Consequently, $U \S_{R_2} V$ but $U \not\S_{R_1} V$, which implies that $S_{R_2} \nsubseteq S_{R_1}$.

Suppose $S_1 \subseteq S_2$. Let $x \in \Ult(A)$ and $y \in \Ult(B)$ such that $x \R_{S_2} y$. Then $S_1[x] \subseteq S_2[x] \subseteq y$, and so $x \R_{S_1} y$. Thus, $R_{S_2} \subseteq R_{S_1}$. Conversely, assume that $S_1 \nsubseteq S_2$. Let $a \in A$ and $b \in B$ be such that $a \S_1 b$ but $a \not\S_2 b$. Since $S_2$ is a subordination, $S_2^{-1}[b]$ is an ideal of $A$ that $a \notin S_2^{-1}[b]$. By the ultrafilter lemma, there is $x \in \Ult(A)$ such that $a \in x$ and $x \cap S_2^{-1}[b] = \varnothing$. Therefore, $b \notin S_2[x]$ and $S_2[x]$ is a filter of $B$ because $S_2$ is a subordination. Thus, there is $y \in \Ult(B)$ such that $b \notin y$ and $S_2[x] \subseteq y$. Since $a \in x$, $b \in y$, and $a \S_1 b$, we have $S_1[x] \nsubseteq y$. Consequently, $x \R_{S_2} y$ but $x \not\R_{S_1} y$, which implies that $R_{S_2} \nsubseteq R_{S_1}$.

	The functors $\Clop$ and $\Ult$ commute with the daggers on $\StoneR$ and $\BAS$.
	Indeed, let $R\colon X\to Y$ be a morphism in $\StoneR$, $U\in\Clop(X)$, and $V\in\Clop(Y)$. 
	Then $\conv{R}[V] \subseteq U$ iff $R[X \setminus U] \subseteq Y \setminus V$. Therefore, $S_{\conv{R}}=\inv{S_R}$.	
	Also, let $S \colon A \to B$ be a morphism in $\baS$, $x\in \Ult(A)$, and $y\in \Ult(B)$. Then $\inv{S}[y] \subseteq x$ iff $S[x] \subseteq y$. Thus, $R_{\inv{S}} = \convrs{R_S}$.

Since $\StoneR$ is an allegory and
the functors $\Clop$ and $\Ult$ are quasi-inverses of each other that preserve and reflect the orders on the hom-sets and commute with the daggers, it follows that $\BAS$ is also an allegory and $\Clop$ and $\Ult$ are morphisms of allegories. Thus, $\StoneR$ and $\BAS$ are equivalent as allegories.
\end{proof}

\begin{remark}
Each hom-set $\hom_{\StoneR}(X,Y)$ is a complete lattice because it is the set of closed subsets of $X \times Y$. Thus, \cref{lem:StoneR and BAS allegories} implies that each hom-set $\hom_{\BAS}(A,B)$ is also a complete lattice.
In \cite{ABC22b} we give an explicit description of meets and joins of subordinations.
\end{remark}

\begin{remark}\label{rem:isos in StoneR and BAS}
A closed relation $R \colon X \to Y$ is an isomorphism in $\StoneR$ iff it is a homeomorphism. To see this, the inverse of $R$ in $\StoneR$ is a closed relation $Q \colon Y \to X$ such that $Q \circ R=\id_X$ and $R \circ Q=\id_Y$.
Since $\id_X$ and $\id_Y$ are identities, $R$ is a bijective function and $Q$ is its inverse. Since a function that is closed as a relation is a continuous function, we conclude that $R$ is a homeomorphism.
Therefore, by \cref{cor: StoneR=BAS}, two boolean algebras are isomorphic in $\BAS$ iff they are isomorphic as boolean algebras. 
\end{remark}

\section{Further lifting to closed relations between compact Hausdorff spaces}\label{sec:lifting equivalences}

The goal of this section is to lift the equivalence between $\StoneR$ and $\BAS$ to the category $\KHausR$ of compact Hausdorff spaces and closed relations. Guided by the fact that $\KHaus$ is the exact completion of $\Stone$ (see, e.g., \cite[Thm.~8.1]{MR20}), we think of compact Hausdorff spaces as the quotients of Stone spaces by closed equivalence relations. We will obtain a category equivalent to $\KHausR$ by splitting closed equivalence relations in $\StoneR$, a construction that we describe in the language of allegories. 

Closed equivalence relations on Stone spaces and their corresponding subordinations on boolean algebras are instances of the notion of an equivalence in an allegory.

\begin{definition}{\cite[p.~198]{FS90}}
Let $C$ be an object of an allegory. A morphism $f \colon C \to C$ is called an \emph{equivalence} if $\id_C \le f$ (\emph{reflexivity}), $f = f^\dagger$ (\emph{symmetry}), and $ff \le f$ (\emph{transitivity}).
\end{definition}

It is immediate to see that a closed relation $R$ on a Stone space is reflexive, symmetric, and transitive as a morphism in $\StoneR$ iff it is reflexive, symmetric, and transitive as a relation. Thus, $R$ is an equivalence in $\StoneR$ iff it is a closed equivalence relation. We have a similar result for subordinations.

\begin{lemma}\label{lem:equivalences StoneR and BAS}
A subordination $S \colon B \to B$ is an equivalence in $\BAS$ iff it satisfies the following conditions:
\begin{enumerate}[label = \normalfont(S\arabic*), ref = S\arabic*, start=5]
	\item \label{S5} $a \S b$ implies $a \le b$.
	\item \label{S6} $a \S b$ implies $\neg b \S \neg a$.
	\item \label{S7} $a \S b$ implies there is $c \in B$ such that $a \S c$ and $c \S b$.
\end{enumerate}
\end{lemma}

\begin{proof}
Since the hom-sets of $\BAS$ are ordered by reverse inclusion, a subordination $S$ on $B$ is a reflexive morphism in $\BAS$ iff $S$ is contained in the partial order $\le$ on $B$, which holds iff $S$ satisfies \eqref{S5}. Since $S=\inv{S}$ iff $S \subseteq \inv{S}$, it follows that $S$ is a symmetric morphism in $\BAS$ iff $S$ satisfies \eqref{S6}. Finally, $S$ is a transitive morphism in $\BAS$ iff $S \subseteq S \circ S$, which means that \eqref{S7} holds. Thus, $S$ is an equivalence in $\BAS$ iff it satisfies conditions \eqref{S5}--\eqref{S7}.

\end{proof}

\begin{definition}
An \emph{$\mathsf{S5}$-subordination} is a subordination $S \colon B \to B$ satisfying conditions \eqref{S5}--\eqref{S7}.
\end{definition}

\begin{remark}
The terminology in the above definition is motivated by the connection to the modal logic $\mathsf{S5}$, which is the logic of relational structures $(X,R)$ such that $R$ is an equivalence relation (see, e.g., \cite[Table~4.1]{BRV01}).
\end{remark}

\begin{definition}
Let $\KHausR$ be the category of compact Hausdorff spaces and closed relations between them. Identity morphisms in $\KHausR$ are identity relations and composition in $\KHausR$ is relation composition.
\end{definition}

Clearly $\StoneR$ is a full subcategory of $\KHausR$.
By arguing as we did in \cref{lem:StoneR and BAS allegories} for $\StoneR$, we obtain:

\begin{lemma}
$\KHausR$ is an allegory.
\end{lemma}

We will show that $\KHausR$ is equivalent to the allegory obtained by splitting the equivalences in $\StoneR$.

\begin{definition}{(\cite[p.~15]{FS90})}
Let $\mathcal{E}$ be a class of idempotent morphisms of a category $\C$, where $e \colon C \to C$ is \emph{idempotent} if $ee=e$.  The category $\Split(\mathcal{E})$ obtained by \emph{splitting $\mathcal{E}$} is defined as follows. The objects of $\Split(\mathcal{E})$ are pairs $(C,e)$ where $C \in \C$ and $e \in \mathcal{E}$ with $e \colon C \to C$. A morphism $f \colon (C,e) \to (C',e')$ in $\Split(\mathcal{E})$ is a morphism $f \colon C \to C'$ in $\C$ such that $fe=f=e'f$.
\[
\begin{tikzcd}
	C \arrow{r}{f} \arrow{rd}{f} \arrow[swap]{d}{e} & C' \arrow{d}{e'}\\
	C \arrow[swap]{r}{f} & C'
\end{tikzcd}
\]
When $\mathcal{E}$ is the class of all idempotents of $\C$, then $\Split(\mathcal{E})$ is called the \emph{Karoubi envelope} or \emph{Cauchy completion} of $\C$.
\end{definition}

\begin{proposition}[{\cite[p.~211]{FS90}}]\label{prop:split is allegory}
Let $\A$ be an allegory and $\mathcal{E}$ a class of symmetric idempotent morphisms of $\A$. Then $\Split(\mathcal{E})$ inherits the structure of an allegory from $\A$. In particular, if $\mathsf{Eq}_{\A}$ is the class of all equivalences in $\A$, then $\Split(\Eq_{\A})$ is an allegory.
\end{proposition}

\begin{notation}
\mbox{}\begin{enumerate}
\item To simplify notation, we denote $\Split(\Eq_{\StoneR})$ by $\StoneER$.
Thus, an object of $\StoneER$ is a pair $(X, E)$ where $X$ is a Stone space and $E$ is a closed equivalence relation on $X$. We call such pairs \emph{$\mathsf{S5}$-subordination spaces}. 
A morphism in $\StoneER$ between $(X,E)$ and $(X',E')$ is a closed relation $R \colon X \to X'$ satisfying $R\circ E=R=E'\circ R$.
We call such relations \emph{compatible closed relations}. The composition of compatible closed relations in $\StoneER$ is the standard relation composition and the identity on $(X,E)$ is the relation $E$. 
\item Similarly, we denote $\Split(\Eq_{\BAS})$ by $\SubSfive$, so
 an object of $\SubSfive$ is a pair $(B,S)$ where $B$ is a boolean algebra and $S$ is an $\mathsf{S5}$-subordination on $B$. We call such pairs {\em $\mathsf{S5}$-subordination algebras}.\footnote{Celani~\cite{Cel05} calls these algebras \emph{quasi-monadic algebras} because they generalize Halmos' monadic algebras~\cite{Hal56}.}
A morphism in $\SubSfive$ between $(B,S)$ and $(B',S')$ is a subordination relation $T \colon B \to B'$ satisfying $T\circ S=T=S'\circ T$.
We call such subordinations \emph{compatible subordinations}. The composition of compatible subordinations in $\SubSfive$ is the standard relation composition and the identity on $(B,S)$ is the subordination $S$. 
\end{enumerate}
\end{notation}

By \cref{prop:split is allegory}, $\StoneER$ is an allegory with the hom-sets ordered by inclusion and the dagger defined as in $\StoneR$. In addition, $\SubSfive$ is an allegory with the hom-sets ordered by reverse inclusion and the dagger defined as in $\BAS$.
The following is now immediate because $\StoneR$ and $\BAS$ are equivalent as allegories (see \cref{lem:StoneR and BAS allegories}):

\begin{theorem}\label{thm:StoneER and SubSfive equivalent}
$\StoneER$ and $\SubSfive$ are equivalent allegories.
\end{theorem}

\begin{remark}\label{rem: Clop Ult StoneER and SubSfive}
The functors establishing the equivalence of allegories are obtained by lifting the functors $\Clop$ and $\Ult$ yielding the equivalence between $\StoneR$ and $\BAS$. An $\mathsf{S5}$-subordination space $(X,E)$ is mapped to $(\Clop(X), S_E)$ and a compatible closed relation $R$ to $S_R$. An $\mathsf{S5}$-subordination algebra $(B,S)$ is mapped to $(\Ult(B), R_S)$ and a compatible subordination $T$ to $R_T$.
\end{remark}

We next show that $\KHausR$ and $\StoneER$ are equivalent as allegories. We first observe that each compact Hausdorff space is homeomorphic to a quotient of a Stone space.
Indeed, any compact Hausdorff space is a quotient of its Gleason cover (see, e.g., \cite[pp.~107--108]{Joh82}). The Gleason cover is not the only possible option. In fact, it follows from the universal mapping property of Stone-\v Cech compactifications (see, e.g., \cite[Thm.~3.6.1]{Eng89}) that any compact Hausdorff space $X$ is homeomorphic to a quotient of the Stone-\v Cech compactification of the discrete space with the same underlying set as $X$. There are many other options as well.
Clearly, if a compact Hausdorff space $X$ is homeomorphic to a quotient $Y/E$ of a Stone space $Y$, then $E$ is necessarily closed (see, e.g., \cite[Thm.~3.2.11]{Eng89}). We will use this in the proof of the next theorem.

\begin{theorem}\label{t:StoneER-KHausR-equivalent}
$\KHausR$ and $\StoneER$ are equivalent allegories. 
\end{theorem}

\begin{proof}
Define $\Qfunc \colon \StoneER \to \KHausR$ as follows: each $(X,E) \in \StoneER$ is sent to the quotient space $X / E$, which is well known to be compact Hausdorff (see, e.g., \cite[Thm.~3.2.11]{Eng89}); and each morphism $R \colon (X,E) \to (X',E')$ in $\StoneER$ is mapped to the morphism $\Qfunc(R) \colon \Qfunc(X,E) \to \Qfunc(X',E')$ in $\KHausR$ given by $\Qfunc(R)=\pi' \circ R \circ \conv{\pi}$, where $\pi \colon X \to X/E$ and $\pi' \colon X' \to X'/E'$ are the projections onto the quotients.

We show that $\Qfunc$ preserves identities and compositions. Note that if $\pi \colon X \to X/E$ is the projection onto the quotient, then $E=\conv{\pi} \circ \pi$ and $\pi \circ \conv{\pi}= \id_{X/E}$.
Therefore, 
\[
\Qfunc(\id_{(X,E)})=\Qfunc(E)=\pi \circ E \circ \conv{\pi} = \pi \circ \conv{\pi} \circ \pi \circ \conv{\pi} =\id_{X/E}.
\] 
Let $R_1 \colon (X_1,E_1) \to (X_2,E_2)$ and $R_2 \colon (X_2,E_2) \to (X_3,E_3)$ be compatible closed relations. Then
\begin{align*}
\Qfunc(R_2) \circ \Qfunc(R_1) & = (\pi_3 \circ R_2 \circ \conv{\pi_2}) \circ (\pi_2 \circ R_1 \circ \conv{\pi_1}) = \pi_3 \circ R_2 \circ E_2 \circ R_1 \circ \conv{\pi_1}\\
 & = \pi_3 \circ R_2 \circ R_1 \circ \conv{\pi_1} = \Qfunc(R_2 \circ R_1).
\end{align*}
We prove that $\Qfunc$ is full, faithful, and essentially surjective.
Let $R' \colon X/E \to X'/E'$ be a closed relation, then $R= \conv{(\pi')} \circ R' \circ \pi$ is a compatible closed relation such that $\Qfunc(R)=R'$ because
\[
\Qfunc(R) =\Qfunc( \conv{(\pi')} \circ R' \circ \pi) = \pi' \circ \conv{(\pi')} \circ R' \circ \pi \circ \conv{\pi} = \id_{X'/E'} \circ R' \circ \id_{X/E} = R'.
\]
Therefore, $\Qfunc$ is full.
If $R \colon (X,E) \to (X',E')$ is a compatible closed relation, then $R=\conv{(\pi')} \circ \Qfunc(R) \circ \pi$ because
\[
\conv{(\pi')} \circ \Qfunc(R) \circ \pi = \conv{(\pi')} \circ \pi' \circ R \circ \conv{\pi}  \circ \pi = E' \circ R \circ E = R.
\]
Thus, if $R,R' \colon (X,E) \to (X',E')$ are two compatible closed relations such that $\Qfunc(R)=\Qfunc(R')$, then $R=\conv{(\pi')} \circ \Qfunc(R) \circ \pi = \conv{(\pi')} \circ \Qfunc(R') \circ \pi = R'$. This shows that $\Qfunc$ is faithful.
As we pointed out before the theorem, if $Y$ is a compact Hausdorff space, then there exists $(X,E) \in \StoneER$ such that $Y$ is homeomorphic to $X/E$. It follows that $\Qfunc$ is essentially surjective, hence an equivalence \cite[p.~93]{Mac98}. It remains to show that $\Qfunc$ is a morphism of allegories.

Let $R_1,R_2 \colon (X,E) \to (X',E')$ be compatible closed relations. If $R_1 \subseteq R_2$, then 
\[
\Qfunc(R_1) = \pi' \circ R_1 \circ \conv{\pi} \subseteq \pi' \circ R_2 \circ \conv{\pi} = \Qfunc(R_2).
\] 
Conversely, if $\Qfunc(R_1) \subseteq \Qfunc(R_2)$, then 
\[
R=\conv{(\pi')} \circ \Qfunc(R_1) \circ \pi \subseteq \conv{(\pi')} \circ \Qfunc(R_2) \circ \pi = R_2.
\] 
Therefore, $\Qfunc$ preserves and reflects the inclusion of relations, and so it yields an order-isomorphism on hom-sets. Finally, since $\conv{(-)}$ is a contravariant endofunctor,
$\Qfunc(\conv{R})=\pi \circ \conv{R} \circ \conv{(\pi')}$ is the converse relation of $\pi' \circ R \circ \conv{\pi} = \Qfunc(R)$. Thus, $\Qfunc$ commutes with the dagger, and hence is a morphism of allegories.
\end{proof}

\begin{remark}\label{rem:equiv from exact completion}
The following is a more categorical approach to \cref{t:StoneER-KHausR-equivalent}. If $\A$ is an allegory, then a morphism $f \colon C \to D$ in $\A$ is a \emph{map} if $f^\dagger f \le \id_C$ and $\id_D \le f f^\dagger$ (see \cite[p.~199]{FS90}). The wide subcategory of $\A$ consisting of maps is denoted by $\Map(\A)$.

If $\mathsf{C}$ is a regular category, then we can define the allegory $\Rel(\mathsf{C})$ whose objects are the same as $\mathsf{C}$ and whose morphisms are internal relations, where an \emph{internal relation} $R \colon C \to D$ is a subobject of $C \times D$ (see \cite[Sec.~A3.1]{Joh02}). Both $\KHaus$ and $\Stone$ are regular categories and their internal relations correspond to closed relations. Therefore, $\KHausR$ and $\Rel(\KHaus)$ are isomorphic allegories, and so are $\StoneR$ and $\Rel(\Stone)$. Thus, the allegories $\Split(\Eq_{\Rel(\Stone)})$ and $ \Split(\Eq_{\StoneR})$ are isomorphic. 

If $\mathsf{C}$ is a regular category, then $\Map(\Split(\Eq_{\Rel(\mathsf{C})}))$ is the effective reflection of $\mathsf{C}$ in the category of regular categories \cite[p.~213]{FS90}.
This is also known as the exact completion or ex/reg completion.\footnote{This should not be confused with the effective reflection of $\mathsf{C}$ in the category of lex categories, also known as ex/lex completion (see \href{https://ncatlab.org/nlab/show/regular+and+exact+completions}{https://ncatlab.org/nlab/show/regular+and+exact+completions}).} Roughly speaking, the exact completion is obtained by closing under quotients.
Since $\KHaus$ is the exact completion of $\Stone$ (see, e.g., \cite[Thm.~8.1]{MR20}), it follows that $\KHaus$ is equivalent to $\Map(\Split(\Eq_{\Rel(\Stone)}))$. Therefore, $\Rel(\KHaus)$ and $\Rel(\Map(\Split(\Eq_{\Rel(\Stone)})))$ are equivalent allegories (see, e.g., \cite[p.~204]{FS90}). By \cite[A3.2.10, A3.3.4(i), A3.3.9(ii)]{Joh02}, the allegories
$\Rel(\Map(\Split(\Eq_{\Rel(\Stone)})))$ and $\Split(\Eq_{\Rel(\Stone)})$ are isomorphic. 
Thus, we have the following chain of equivalences of allegories:
	\[
		\KHausR \cong \Rel(\KHaus) \simeq \Rel(\Map(\Split(\Eq_{\Rel(\Stone)}))) \cong \Split(\Eq_{\Rel(\Stone)}) \cong \Split(\Eq_{\StoneR}) = \StoneER,
	\]
where $\cong$ stands for isomorphism and $\simeq$ for equivalence of allegories.
\end{remark}

As an immediate consequence of \cref{t:StoneER-KHausR-equivalent,thm:StoneER and SubSfive equivalent} we obtain:

\begin{corollary}
$\KHausR$, $\StoneER$, and $\SubSfive$ are equivalent as allegories.
\end{corollary}

We conclude this section by characterizing isomorphisms in $\StoneER$ and $\SubSfive$.

\begin{remark}\label{rem:isos in StoneER}\label{r:iso}
	\mbox{}\begin{enumerate}[label=\normalfont(\arabic*), ref = \arabic*]
		\item \label{i:iso}\label{i:E-iso} By \cref{t:StoneER-KHausR-equivalent}, $\Qfunc$ is full and faithful. Therefore, a morphism $R \colon (X_1, E_1) \to (X_2, E_2)$ in $\StoneER$ is an isomorphism iff  $\Qfunc(R) \colon X_1/E_1 \to X_2/E_2$ is a homeomorphism. 
		Thus, $(X_1,E_1)$ and $(X_2,E_2)$ are isomorphic in $\StoneER$ iff $X_1/E_1$ and $X_2/E_2$ are homeomorphic.
		\item \label{i:E-inverse} 
Let $R \colon (X_1, E_1) \to (X_2, E_2)$ be an isomorphism in $\StoneER$. Since inverses in an allegory are daggers (see \cite[p.~199]{FS90}) and $R^\dagger=\conv{R}$ in $\StoneER$ (see \cref{lem:StoneR allegory} and the paragraph before \cref{thm:StoneER and SubSfive equivalent}), it follows that the inverse of $R$ in $\StoneER$ is $\conv{R}$.
\item \label{i:isos StoneER and structure preserving} 	Not every isomorphism in $\StoneER$ is a structure-preserving bijection; two objects $(X_1, E_1)$ and $(X_2, E_2)$ may be isomorphic even if $X_1$ and $X_2$ do not share the same cardinality.
	However, the converse is true.
If $f \colon X_1 \to X_2$ is a structure-preserving bijection, then $f$ is a homeomorphism such that $x \E_1 y$ iff $f(x) \E_2 f(y)$. Therefore, $f$ and $f^{-1}=\conv{f}$ are compatible closed relations, and hence $f$ is an isomorphism in  $\StoneER$. 
\item\label{rem:isos in StoneER:item4} A similar result is true for $\SubSfive$: two objects $(B_1, S_1)$ and $(B_2, S_2)$ may be isomorphic even if $B_1$ and $B_2$ do not share the same cardinality.
	However, 
	every structure-preserving bijection gives rise to an isomorphism in $\SubSfive$.
	Indeed, let $(B_1,S_1),(B_2, S_2) \in \SubSfive$ and suppose there is a boolean isomorphism $f \colon B_1 \to B_2$ such that for all $a,b \in B_1$ we have $a \S_1 b$ iff $f(a) \S_2 f(b)$.
	Similarly to \cref{rem:isos in StoneR and BAS}, an isomorphism between $(B_1,S_1)$ and $(B_2, S_2)$ is given by the relation $T \colon B_1 \to B_2$ defined by 
	\[
	a \T b \iff f(a) \S_2 b \iff a \S_1 f^{-1}(b),
	\]
	and its inverse in $\SubSfive$ is the relation $Q \colon B_2 \to B_1$ defined by
	\[
	b \Q a \iff b \S_2 f(a) \iff f^{-1}(b) \S_1 a.
	\]
	By (\ref{S6}),
	\[
	b \mathrel{\inv{T}} a \iff \neg a \T \neg b \iff f(\neg a) \S_2 \neg b \iff b \S_2 f(a) \iff b \Q a.
	\]
	Thus, every structure-preserving bijection $f \colon B_1 \to B_2$ gives rise to an isomorphism $T \colon B_1 \to B_2$ in $\SubSfive$ whose inverse is $\inv{T} \colon B_2 \to B_1$.
	\end{enumerate}
\end{remark}

\section{De Vries algebras and Gleason spaces} \label{sec: main equivalences}

As we saw in \cref{sec:lifting equivalences}, $\KHausR$ is equivalent to both $\StoneER$ and $\SubSfive$. In this section we introduce two important full subcategories $\GleR$ of $\StoneER$ and $\devS$ of $\SubSfive$. The objects of $\GleR$ are Gleason spaces and those of $\devS$ are de Vries algebras. We prove that $\GleR$ and $\devS$ are also equivalent to $\KHausR$. In \cref{sec: 5} we will see that, unlike $\StoneER$ and $\SubSfive$, isomorphisms in $\GleR$ and $\devS$ are structure-preserving bijections.

For a compact Hausdorff space $X$, let $g_X \colon \widehat{X} \to X$ be the Gleason cover of $X$ (see, e.g., \cite[pp.~107--108]{Joh82}).
Let $E$ be the equivalence relation on $\widehat{X}$ given by 
\begin{equation}\label{eq:def E Gleason}
x \E y \iff g_X(x) = g_X(y).\tag{$\ast$}
\end{equation}
Since $X$ is homeomorphic to the quotient space $\widehat{X}/E$, in a certain sense, we can identify $X$ with the pair $(\widehat{X}, E)$.
This was made precise in \cite{BezhanishviliGabelaiaEtAl2019}, where an equivalence is exhibited between $\KHausR$ and the full subcategory of $\StoneER$ that we next define.

\begin{definition} \mbox{}
	\begin{enumerate}
		\item \cite[p.~368]{Eng89} A topological space is \emph{extremally disconnected} if the closure of each open set is open. 
		\item \cite[Def.~6.6]{BBSV17}
		A \emph{Gleason space} is an object $(X,E)$ of $\StoneER$ such that $X$ is extremally disconnected and $E$ is \emph{irreducible} (i.e., if $F$ is a proper closed subset of $X$, then so is $E[F]$). 
		\item \cite[Def.~3.6]{BezhanishviliGabelaiaEtAl2019} Let $\GleR$ be the full subcategory of $\StoneER$ whose objects are Gleason spaces.
	\end{enumerate}
\end{definition}

We next define the category $\devS$.

\begin{definition}
\hfill
\begin{enumerate}
\item \cite[Def.~3.2]{Bez10} A {\em de Vries algebra} is an $\mathsf{S5}$-subordination algebra $(B,S)$ such that $B$ is a complete boolean algebra and $S$ satisfies the following axiom:
	\begin{enumerate}[label = (S\arabic*), ref = S\arabic*]
		\setcounter{enumii}{7}
		\item \label{S8} If $a \neq 0$, then there is $b \neq 0$ such that $b \S a$.
	\end{enumerate}
\item Let $\devS$ be the full subcategory of $\SubSfive$ consisting of de Vries algebras.  
\end{enumerate}
\end{definition}

\begin{remark}\label{rem:prec de Vries}
A relation $S$ on a boolean algebra $B$ satisfying \eqref{S1}--\eqref{S8} is called a \emph{de Vries proximity} on $B$ and is usually denoted by $\prec$.
\end{remark}

\begin{lemma}
$\GleR$ and $\devS$ are allegories.
\end{lemma}

\begin{proof}
By~\cite[Examples~A.3.2.2(c)]{Joh02}, every full subcategory of an allegory inherits the allegory structure.
\end{proof}

\begin{theorem} \label{i:restricts}
 The equivalence between $\StoneER$ and $\SubSfive$ restricts to an equivalence of allegories between $\GleR$ and $\devS$.
\end{theorem}

\begin{proof}
Let $B \in \ba$. By Stone duality, $B$ is complete iff $\Ult(B)$ is extremally disconnected. Let $S$ be an $\sf{S5}$-subordination on $B$ and $R_S$ the corresponding closed equivalence relation on $\Ult(B)$. By \cite[Lem.~6.3]{BBSV17}, $R_S$ is irreducible iff $S$ satisfies \eqref{S8}. Thus, the restriction of $\Clop\colon \StoneER \to \SubSfive$ and $\Ult\colon \SubSfive \to \StoneER$ gives the desired equivalence.
\end{proof}

For a compact Hausdorff space $X$, let $\G(X)=(\widehat{X},E)$, where $\widehat{X}$ is the Gleason cover of $X$ and $E$ the corresponding equivalence relation defined in \eqref{eq:def E Gleason}. 
For a closed relation $R \colon X \to X'$, let $\G(R) \colon \G(X) \to \G(X')$ be given by $\G(R)=\convgxprime{g_{X'} } \circ R\circ g_X$. This defines a functor $\G \colon \KHausR \to \StoneER$, which is a quasi-inverse of $\Qfunc$:

\begin{theorem} \label{prop:replete}
	Each object $(X, E)$ of $\StoneER$ is isomorphic in $\StoneER$ to the Gleason space $\G(X/E)$.
	Thus, $\G$ is a quasi-inverse of $\Qfunc$ and the inclusion of $\GleR$ into $\StoneER$ is an equivalence of allegories.
\end{theorem}

\begin{proof}
	Let $(X,E) \in \StoneER$ and let $(X',E')=\G(X/E)$. Then $(X',E')\in\GleR$ and $(X,E)$ is isomorphic to $(X',E')$ by \cref{rem:isos in StoneER}\eqref{i:E-iso}.
	Thus, the inclusion of $\GleR$ into $\StoneER$ is full, faithful, and essentially surjective, hence an equivalence of allegories (because it is a morphism of allegories by~\cite[Examples~A.3.2.2(c)]{Joh02}).
	Since any Gleason space is an object of $\StoneER$, what we observed above implies that any $(X,E) \in \GleR$ is isomorphic to $\G(X/E) = \G(\Qfunc(X,E))$.
	Straightforward computations show that the isomorphisms $(X,E) \cong \G(\Qfunc(X,E))$ for $(X,E) \in \GleR$ and $\Qfunc(\G(X)) \cong X$ for $X \in \KHausR$ yield natural isomorphisms $\id_{\GleR} \cong \G \circ \Qfunc$ and $\id_{\KHausR} \cong \Qfunc \circ \G$. Thus, $\G$ is a quasi-inverse of $\Qfunc$.
\end{proof}

\begin{corollary} \label{cor:all-equivalent}
$\KHausR$, $\GleR$, and $\devS$ are equivalent allegories.
\end{corollary}

As a consequence, we arrive at the following diagram of equivalences of allegories that commutes up to natural isomorphisms. 
\[
\begin{tikzcd}
	&[-10pt] \StoneER \arrow[r, shift left= 0.2em]{}{\Clop}  \arrow[ld, shift left= 0.2em]{}{\Qfunc} &[50pt] \SubSfive \arrow[l, shift left = 0.2em]{}{\Ult}\\
	\KHausR \arrow[ur, shift left= 0.2em]{}{\G} \arrow[dr, shift left= 0.2em]{}{\G} & & \\
	& \GleR \arrow[uu, hookrightarrow] \arrow[ul, shift left= 0.2em]{}{\Qfunc} \arrow[r, shift left= 0.2em]{}{\Clop} & \devS \arrow[uu, hookrightarrow] \arrow[l, shift left= 0.2em]{}{\Ult}
\end{tikzcd}
\]

\begin{remark}\label{rem:open problem}
That $\KHausR$ and $\GleR$ are equivalent categories was first established in \cite[Thm.~3.13]{BezhanishviliGabelaiaEtAl2019}.
By \cref{cor:all-equivalent}, $\KHausR$ is equivalent to $\devS$. This resolves the problem raised in \cite[Rem.~3.14]{BezhanishviliGabelaiaEtAl2019} to find a generalization of the category of de Vries algebras that is (dually) equivalent to $\KHausR$. The key is to work with the category $\devS$ in which functions between de Vries algebras are replaced by relations. 
\end{remark}

We conclude this section by providing an explicit description of the functor from $\KHausR$ to $\devS$ yielding the equivalence, which is a direct generalization of the regular open functor of de Vries duality.

For a compact Hausdorff space $X$, let $\RO(X)$ be the complete boolean algebra of regular open subsets of $X$. We recall that joins in $\RO(X)$ are computed as $\bigvee_{i}U_i = \int(\cl(\bigcup_i U_i))$ and the negation is computed as $\lnot U = \int(X \setminus U)$. Similarly, if $\RC(X)$ is the complete boolean algebra of regular closed subsets of $X$, then  
meets in $\RC(X)$ are computed as $\bigwedge_{i}F_i = \cl(\int(\bigcap_i F_i))$ and the negation is computed as $\lnot F = \cl(X \setminus F)$.

Parts of the next lemma are known, but it is convenient to collect all the relevant isomorphisms together.

\begin{lemma}\label{lem:iso RO(X) RC(X) Clop(GX)}
	Let $X$ be a compact Hausdorff space and $g_X \colon \widehat{X} \to X$ its Gleason cover.
	The boolean algebras $\RO(X)$, $\RC(X)$, and $\Clop(\widehat{X})$ are isomorphic, with the isomorphisms given by\\
	\begin{minipage}{0.33\textwidth}
		\begin{align*}
			\RO(X) & \longleftrightarrow \RC(X)\\
			U & \longmapsto \cl(U)\\
			\int(F) & \longmapsfrom F,
		\end{align*}
	\end{minipage}%
	\begin{minipage}{0.34\textwidth}
		\begin{align*}
			\Clop(\widehat{X}) & \longleftrightarrow \RC(X)\\
			V & \longmapsto g_X[V]\\
			\cl(g_X^{-1}[\int (F)]) & \longmapsfrom F,
		\end{align*}
	\end{minipage}%
	\begin{minipage}{0.33\textwidth}
		\begin{align*}
			\Clop(\widehat{X}) & \longleftrightarrow \RO(X)\\
			V & \longmapsto \int(g_X[V])\\
			\cl(g_X^{-1}[U]) & \longmapsfrom U.
		\end{align*}
	\end{minipage}%
\end{lemma}

\begin{proof}(Sketch) That $\cl \colon \RO(X) \to \RC(X)$
	and $\int \colon \RC(X) \to \RO(X)$ are inverses of each other is an immediate consequence of the definition of regular open and regular closed sets.
	
	Since $g_X \colon \widehat{X} \to X$ is an onto irreducible map, the direct image function $g_X[-] \colon \RC(\widehat{X}) \to \RC(X)$ that maps $F \in \RC(\widehat{X})$ to $g_X[F]$ is a boolean isomorphism (see~\cite[p.~454]{PW88}).
	We note that $\RC(\widehat{X}) = \Clop(\widehat{X})$ because $\widehat{X}$ is extremally disconnected.
	It follows from the proof in \cite[p.~454]{PW88} that the inverse of $g_X[-]$ is given by mapping each $F \in \RC(X)$ to $\cl (g_X^{-1}[\int (F)])$.
	
	By composing the isomorphism between $\RO(X)$ and $\RC(X)$ with the isomorphism between $\RC(X)$ and $\Clop(\widehat{X})$, we obtain the isomorphism between $\RO(X)$ and $\Clop(\widehat{X})$ described in the statement.
\end{proof}

\begin{definition} \label{d:rofunc}
Let $\rofunc \colon \KHausR \to \devS$ be defined by associating with each compact Hausdorff space $X$ the de Vries algebra $(\RO(X), \prec)$ of regular open subsets of $X$ (with the de Vries proximity defined by $U \prec V$ iff $\cl(U) \subseteq V$) and with each closed relation $R \colon X \to X'$ the subordination $\rofunc(R) \colon \RO(X) \to \RO(X')$ given by
\begin{center}
$U \mathrel{\rofunc(R)} V \iff R[\cl(U)] \subseteq V$.
\end{center}
\end{definition}

It is straightforward to see that $\rofunc$ is a well-defined functor.
Recalling the functors $\G \colon \KHausR \to \GleR$ and $\Clop \colon \GleR \to \devS$, we obtain:

\begin{theorem}\label{thm: equivalence}
	The functor $\rofunc$ is naturally isomorphic to $\Clop \circ \G$. 
\end{theorem}

\begin{proof}
	By Lemma~\ref{lem:iso RO(X) RC(X) Clop(GX)}, for every compact Hausdorff space $X$, the function $\eta_X \colon \RO(X) \to \Clop(\G X)$ that maps $U$ to $\cl(g_X^{-1}[U])$ is a boolean isomorphism.	
	
If $R \colon X \to X'$ is a closed relation, then $\G R \colon \G X \to \G X'$ is given by $\G R = \convgxprime{g_{X'} } \circ R \circ g_X$, 
so $x \mathrel{\G R} x'$ iff $g_X(x) \R g_{X'}(x')$.
Moreover, $\Clop(\G R)$ is the $\devS$-morphism
$S_{\G R} \colon \Clop(\G X) \to \Clop(\G X')$ given by
$U \S_{\G R} V$ iff $(\G R)[U] \subseteq V$. Consequently, for all $U \in \RO(X)$ and $V \in \RO(X')$ we have
\begin{align*}
\eta_X(U) \mathrel{(\Clop\,\G R)} \eta_{X'}(V) & \iff \G R[\eta_X(U)] \subseteq \eta_{X'}(V)\\
& \iff \G R[\eta_X(U)] \cap (\G X' \setminus \eta_{X'}(V)) = \varnothing\\
& \iff g_{X'}^{-1}[R[g_X[\eta_X(U)]]] \cap (\G X' \setminus \eta_{X'}(V)) = \varnothing\\
& \iff R[g_X[\eta_X(U)]] \cap g_{X'}[\G X' \setminus \eta_{X'}(V)] = \varnothing.
\end{align*}
Since $U \in \RO(X)$ and $\cl(U) \in \RC(X)$, it follows from Lemma~\ref{lem:iso RO(X) RC(X) Clop(GX)}
that
\begin{align*}
R[g_X[\eta_X(U)]] = R[g_X[\cl(g_X^{-1}[\int (\cl(U))])]] = R[\cl(U)].
\end{align*}
Because $\eta_{X'}$ is an isomorphism,
\begin{align*}
g_{X'}[\G X' \setminus \eta_{X'}(V)]
= g_{X'}[\eta_{X'}(\int(X' \setminus V))] = g_{X'}[\cl(g_{X'}^{-1}[\int(X' \setminus V)])] = X' \setminus V,
\end{align*}
where the second equality follows from Lemma~\ref{lem:iso RO(X) RC(X) Clop(GX)}
since $X' \setminus V \in \RC(X')$.
Therefore,
\begin{align*}
\eta_X(U) \mathrel{(\Clop\,\G R)} \eta_{X'}(V) \iff R[\cl(U)] \cap (X' \setminus V) = \varnothing \iff R[\cl(U)] \subseteq V.
\end{align*}
	
When $X = X'$, taking $R$ to be the identity $1_X$ on $X$, 
 we have 
\begin{equation*}
\eta_X(U) \mathrel{(\Clop \, \G 1_X)} \eta_X(V) \iff U \prec V
\end{equation*} 
 for all $U, V \in \RO(X)$.
Thus, the isomorphism $\eta_X \colon \RO(X) \to \Clop(\G X)$ is a structure-preserving bijection. By Remark~\ref{r:iso}\eqref{rem:isos in StoneER:item4},
$\RO(X)$ and $\Clop(\G X)$ are isomorphic in $\devS$, and
the isomorphism is given by the relation $S_X \colon \RO(X) \to \Clop(\G{X})$ defined by 
\begin{align*}
U \S_X V \iff \eta_X(U) \mathrel{(\Clop \, \G 1_X)} V \iff U \prec \eta_X^{-1}(V)
\end{align*}
 (see Remark~\ref{r:iso}\eqref{rem:isos in StoneER:item4} and Lemma~\ref{lem:iso RO(X) RC(X) Clop(GX)}).

To prove naturality, we show that the following diagram commutes for every morphism $R \colon X \to X'$ in $\KHaus$.
\[
\begin{tikzcd}
\RO(X) \arrow{r}{S_X} \arrow[swap]{d}{\rofunc(R)} & \Clop(\G{X}) \arrow{d}{\Clop \, \G R}\\
\RO(X') \arrow{r}{S_{X'}} & \Clop(\G{X'})
\end{tikzcd}
\]
Let $U \in \RO(X)$ and $V \in \Clop(\G{X'})$.
Since $\Clop\,\G R$ and $\rofunc(R)$ are compatible subordinations, we have
\begin{align*}
U \mathrel{(S_{X'} \circ \rofunc(R))} V & \iff \exists C \in \RO(X') : U \mathrel{\rofunc(R)} C \S_{X'} V \\	
& \iff \exists C \in \RO(X') : U \mathrel{\rofunc(R)} C \prec \eta_{X'}^{-1}(V) \\	
& \iff U \mathrel{\rofunc(R)} \eta_{X'}^{-1}(V) \\
& \iff \eta_X(U) \mathrel{(\Clop \, \G R)} V \\
& \iff \exists D \in \Clop(\G X) : \eta_X(U) \mathrel{(\Clop \, \G 1_X)} D \mathrel{(\Clop \, \G R)} V \\
& \iff \exists D \in \Clop(\G X) : U \S_X D \mathrel{(\Clop \, \G R)} V\\
& \iff U \mathrel{((\Clop\,\G R) \circ S_X)} V. \qedhere
\end{align*}
\end{proof}

We thus obtain:

\begin{corollary} \label{c:equiv-explicit}
	The functor $\rofunc \colon \KHausR \to \devS$ is an equivalence of allegories.
\end{corollary}

\begin{proof}
	By Theorem~\ref{thm: equivalence}, $\rofunc$ is naturally isomorphic to the composition $\Clop \circ \G$, each of which is an equivalence of allegories (see \cref{rem: Clop Ult StoneER and SubSfive,prop:replete}).
\end{proof}

\section{Isomorphisms in $\devS$ and $\GleR$}\label{sec: 5}

It follows from \cref{rem:isos in StoneER} that isomorphisms in $\SubSfive$ and $\StoneER$ are not structure-preserving bijections. 
In this section we show that in $\devS$ and $\GleR$ isomorphisms become structure-preserving bijections, thus making the latter categories more convenient to work with. 

For a subset $E$ of a boolean algebra $B$, we write $U(E)$ and $L(E)$ for the sets of upper and lower bounds of $E$, respectively. 
We will freely use the fact that in a de Vries algebra $(A,S)$ we have $a=\bigwedge S[a]=\bigvee S^{-1}[a]$ for every $a \in A$.

\begin{lemma} \label{lem:two-sorted-irreducibility} \label{lem:normal-pair}
Let $(A, S_A)$ and $(B, S_B)$ be isomorphic objects in $\SubSfive$ with $T \colon A \to B$ and $Q \colon B \to A$ inverses of each other. Suppose that $(A, S_A)$ is a de Vries algebra.
\begin{enumerate}[label=\normalfont(\arabic*), ref = \arabic*]
\item\label{1} For all $a_1, a_2 \in A$, we have
\[
a_1 \leq a_2 \iff T[a_1] \supseteq T[a_2] \iff Q^{-1}[a_1] \subseteq Q^{-1}[a_2].
\]
\item\label{2} For all $a \in A$, we have
\begin{align*}
Q^{-1}[a] = S_B^{-1}[L(T[a])] \qquad \mbox{and} \qquad T[a] = S_B[U(Q^{-1}[a])].
\end{align*}
\end{enumerate}
\end{lemma}

\begin{proof}
	\eqref{1}. We only prove $a_1 \leq a_2 \iff T[a_1] \supseteq T[a_2]$; the equivalence $a_1 \leq a_2 \iff Q^{-1}[a_1] \subseteq Q^{-1}[a_2]$ is proved similarly.
	The left-to-right implication is immediate.
	For the right-to-left implication, suppose that $T[a_1] \supseteq T[a_2]$.
	Then $QT[a_1] \supseteq QT[a_2]$.
	Since $T$ and $Q$ are inverses of each other, $Q \circ T = S_A$, and hence $S_A[a_1]=QT[a_1] \supseteq QT[a_2]=S_A[a_2]$.
	Thus, $a_1= \bigwedge S_A[a_1] \le \bigwedge S_A[a_2]=a_2$ because $A$ is a de Vries algebra.

	\eqref{2}. We only prove $Q^{-1}[a] = S_B^{-1}[L(T[a])]$; the second equality is proved similarly.
	For the left-to-right inclusion, let $b \in Q^{-1}[a]$, so $b \Q a$.
	Since $Q$ is a compatible subordination, there is $b' \in B$ such that $b \S_B b' \Q a$.
	For every $b'' \in T[a]$ we have $b' \Q a \T b''$, so $b' \S_B b''$, and hence $b' \leq b''$ by (\ref{S5}).
	Therefore, $b' \in L(T[a])$, and so $b \in S_B^{-1}[L(T[a])]$.
	
	For the right-to-left inclusion, let $b \in S_B^{-1}[L(T[a])]$.
	Then there is $b' \in L(T[a])$ such that $b \S_B b'$.
	Since $Q$ and $T$ are inverses of each other, there is $a' \in A$ such that $b \Q a' \T b'$.
	But then $T[a] \subseteq T[a']$ because for every $c \in T[a]$, we have $a' \T b' \leq c$, and so $a' \T c$. 
	Thus, $a' \leq a$ by (\ref{1}). 
	Consequently, $b \Q a' \leq a$, so $b \Q a$, and hence $b \in Q^{-1}[a]$.
\end{proof}

\begin{lemma} \label{lem:T and Q iso in devS}
	Let $(A, S_A)$ and $(B, S_B)$ be isomorphic objects in $\devS$ with
	$T \colon A \to B$ and $Q \colon B \to A$ inverses of each other. 
	For all $a \in A$ and $b \in B$,
\begin{enumerate}[label=\normalfont(\arabic*), ref = \arabic*]
\item\label{lem:T and Q iso in devS:item1} $T[a] = S_B\mleft[\bigvee Q^{-1}[a]\mright] = S_B\mleft[\bigwedge T[a]\mright];$
\item\label{lem:T and Q iso in devS:item2} $Q[b] = S_A\mleft[\bigvee T^{-1}[b]\mright] = S_A\mleft[\bigwedge Q[b]\mright];$
\item\label{lem:T and Q iso in devS:item3} $T^{-1}[b] = S_A^{-1}\mleft[\bigwedge Q[b]\mright] = S_A^{-1}\mleft[\bigvee T^{-1}[b]\mright];$
\item\label{lem:T and Q iso in devS:item4} $Q^{-1}[a] = S_B^{-1}\mleft[\bigwedge T[a]\mright] = S_B^{-1}\mleft[\bigvee Q^{-1}[a]\mright].$
\end{enumerate}	
\end{lemma}

\begin{proof}
We only prove the first equality of \eqref{lem:T and Q iso in devS:item4} and the second equality of \eqref{lem:T and Q iso in devS:item1}. The rest are proved similarly.
To see that $Q^{-1}[a] = S_B^{-1}[\bigwedge T[a]]$, by Lemma~\ref{lem:normal-pair}(\ref{2}) it is sufficient to prove that $S_B^{-1}[L(T[a])] = S_B^{-1}[\bigwedge T[a]]$. But this is obvious because $\bigwedge T[a]$ is the greatest lower bound of $T[a]$. 
	
	To see that $S_B\mleft[\bigvee Q^{-1}[a]\mright] = S_B\mleft[\bigwedge T[a]\mright]$, since $(B, S_B)$ is a de Vries algebra, for each $c\in B$, we have $c = \bigvee S_B^{-1}[c]$. Thus, by the first equality of \eqref{lem:T and Q iso in devS:item4},
	\begin{align*}
	S_B\mleft[\bigvee Q^{-1}[a]\mright] = S_B\mleft[\bigvee S_B^{-1}\mleft[\bigwedge T[a]\mright]\mright] = S_B\mleft[\bigwedge T[a]\mright].&\qedhere
	\end{align*}
\end{proof}

\begin{lemma}\label{prop:isos}
	Let $(A, S_A)$ and $(B, S_B)$ be isomorphic objects in $\devS$ with $T \colon A \to B$ and $Q \colon B \to A$ inverses of each other. Define $f \colon A \to B$ and $g \colon B \to A$ by
	\begin{align*}
	f(a)= \bigwedge T[a] \quad \mbox{and}\quad g(b) = \bigwedge Q[b].
	\end{align*}
	Then $f$ and $g$ are structure-preserving bijections that are inverses of each other.
	Moreover, for each $a\in A$ and $b\in B$ we have
	\begin{align*}
	a \T b \mbox{ iff } f(a) \S_B b \quad \mbox{and} \quad b \Q a \mbox{ iff } g(b) \S_A a.
	\end{align*}
\end{lemma}

\begin{proof}
Let $a \in A$. Since $Q$ is a compatible subordination and $T$ is its inverse, by Lemma~\ref{lem:T and Q iso in devS}\eqref{lem:T and Q iso in devS:item1} we have 
\[
Q[f(a)] = Q\mleft[\bigwedge T[a]\mright] = Q S_B\mleft[\bigwedge T[a]\mright] = QT[a] = S_A[a].
\]
Thus, since $S_A$ is a de Vries proximity, 
\begin{align*}
gf(a) = \bigwedge Q[f(a)] = \bigwedge S_A[a] = a.
\end{align*}
A similar proof yields that $fg(b)=b$ for each $b \in B$. 
It is an immediate consequence of Lemma~\ref{lem:two-sorted-irreducibility} that $f$ and $g$ are order-preserving. Therefore, $f,g$ are boolean isomorphisms that are inverses of each other.

We next show that $f$ preserves and reflects $S_A$. That $g$ preserves and reflects $S_B$ is proved similarly. Let $a,a' \in A$. As we saw above, $Q[f(a)] = S_A[a]$. Also, by Lemma~\ref{lem:T and Q iso in devS}\eqref{lem:T and Q iso in devS:item4}, $S_B^{-1}[f(a')]=S_B^{-1}[\bigwedge T[a']]=Q^{-1}[a']$. Therefore,
\begin{align*}
f(a) \S_B f(a') & \iff f(a) \in S_B^{-1}[f(a')] \iff f(a) \in Q^{-1}[a'] \iff a' \in Q[f(a)]\\
& \iff a' \in S_A[a] \iff a \S_A a'.
\end{align*}
Finally, let $a\in A$ and $b\in B$. To see that $a \T b$ iff $f(a) \S_B b$, it is sufficient to observe that Lemma~\ref{lem:T and Q iso in devS}\eqref{lem:T and Q iso in devS:item1} implies $S_B[f(a)]=S_B[\bigwedge T[a]]=T[a]$. A similar reasoning gives that $b \Q a$ iff $g(b) \S_A a$.
\end{proof}

As an immediate consequence of Remark~\ref{r:iso} and Lemma~\ref{prop:isos} we obtain:

\begin{theorem}\label{thm:isos}
Isomorphisms in $\devS$ are given by structure-preserving bijections.
\end{theorem}

\begin{remark}
The above theorem generalizes a similar result for the category $\dev$ of de Vries algebras (see \cite[Prop.~1.5.5]{deV62}).
\end{remark}

\begin{remark} \label{cor:iso}
An analogue of Theorem~\ref{thm:isos} is that isomorphisms in $\GleR$ are homeomorphisms that preserve and reflect the equivalence relation. 
\end{remark}

\section{An alternative approach to de Vries duality}\label{sec: functional}

In this final section we show that the equivalence between $\KHausR$ and $\devS$ restricts to an equivalence between $\KHaus$ and the wide subcategory $\devF$ of $\devS$ whose morphisms satisfy two additional conditions (the superscript $\mathsf{F}$ signifies that morphisms in $\KHaus$ are functions). This provides an alternative of de Vries duality. We finish the paper by giving a direct proof that $\dev$ is dually isomorphic to $\devF$. 
An advantage of $\devF$ over $\dev$ is that composition in $\devF$ is usual relation composition. 

\begin{definition}
We define the category $\StoneEf$ as $\Map(\StoneER)$.
Explicitly, this is the category whose objects are pairs $(X, E)$ where $E$ is a closed equivalence relation on a Stone space $X$, and whose morphisms from $(X_1,E_1)$ to $(X_2,E_2)$ are the compatible closed relations $R \colon X_1 \to X_2$ such that $E_1 \subseteq \conv{R} \circ R$ and $R \circ \conv{R} \subseteq E_2$.
\end{definition}

\begin{proposition} \label{prop:KH-StoneEf}
The categories $\KHaus$ and $\StoneEf$ are equivalent.
\end{proposition}

\begin{proof}
By \cref{t:StoneER-KHausR-equivalent}, $\KHausR$ and $\StoneER$ are equivalent allegories. Therefore, $\Map(\KHausR)$ and $\Map(\StoneER)$ are equivalent categories (see \cite[p.~204]{FS90}).
The result follows since $\KHaus = \Map(\KHausR)$ and $\StoneEf = \Map(\StoneER)$.
\end{proof}

\begin{remark}\label{rem:proof exact completion}
Since 
$\StoneEf = \Map(\StoneER) = \Map(\Split(\Eq_{\StoneR}))$
and $\Map(\Split(\Eq_{\StoneR}))$ is isomorphic to $\Map(\Split(\Eq_{\Rel(\Stone)}))$ (see \cref{rem:equiv from exact completion}),
we have that $\StoneEf$ is the exact completion of $\Stone$. Thus, as a consequence we obtain an alternative proof of the fact that $\KHaus$ is the exact completion of $\Stone$. 
\end{remark}

\begin{definition}
	We define the category $\SubSfivef$ as $\Map(\SubSfive)$.
	Explicitly, this is the category whose objects are pairs $(B, S)$ where $S$ is a $\mathsf{S5}$-subordination on a Boolean algebra $B$, and whose morphisms from $(A,S_A)$ to $(B,S_B)$ are the morphisms $T \colon (A, S_A) \to (B, S_B)$ in $\SubSfive$ satisfying $\inv{T} \circ T \subseteq S_A$ and $S_B \subseteq T \circ \inv{T}$.
\end{definition}

We next give a slightly more explicit description of morphisms in $\SubSfivef$.

\begin{lemma} \label{l:explicit}
	A morphism $T \colon (B_1,S_1) \to (B_2,S_2)$ in $\SubSfive$ is a morphisms in $\SubSfivef$ iff the following two conditions hold.
	\begin{enumerate}[label=\normalfont(\arabic*), ref=\arabic*]
		\item \label{i:total}
		$\forall a \in B_1 \ (a \T 0 \Rightarrow a = 0)$.

		\item \label{i:deterministic} 
		$\forall b_1, b_2 \in B_2 \ (b_1 \S_2 b_2 \Rightarrow \exists a \in B_1 : \neg a \T \neg b_1 \mbox{ and } a \T b_2)$.
	\end{enumerate}
\end{lemma}

\begin{proof}
	It is immediate that \eqref{i:deterministic} is equivalent to $S_2 \subseteq T \circ \inv{T}$.
	Moreover, 
	$\inv{T} \circ T \subseteq S_1$ is equivalent to the following condition:
	\begin{equation} \label{e:tot}
		\forall a_1,a_2 \in B_1\ \forall b \in B_2\ ((a_1 \T b \mbox{ and } \neg a_2 \T \neg b) \Rightarrow a_1 \S_1 a_2). \tag{$\ast \ast$}
	\end{equation}
	We show that \eqref{e:tot} is equivalent to \eqref{i:total}.
	Suppose that \eqref{e:tot} holds, $a \in B_1$, and $a \T 0$.
	Letting $a_1 = a$, $a_2 = 0$, and $b = 0$ in \eqref{e:tot} yields $a \S_1 0$. Therefore, $a = 0$ by \eqref{S5}.
	Conversely, suppose that \eqref{i:total} holds. 
	Let $a_1, a_2 \in B_1$ and $b \in B_2$ such that $a_1 \T b$ and $\lnot a_2 \T \lnot b$.
	Since $T$ is a compatible subordination, from $a_1 \T b$ it follows that there is $a \in B_1$ such that $a_1 \S_1 a \T b$.
	From $\lnot a_2 \T \lnot b$ and $a \T b$ it follows that $(\lnot a_2 \land a) \T (\lnot b \land b) = 0$.
	Therefore, \eqref{i:total} implies $\lnot a_2 \land a = 0$, so $a \leq a_2$.
	Thus, $a_1 \S_1 a \leq a_2$, and hence $a_1 \S_1 a_2$.
\end{proof}

\begin{proposition} \label{prop:StoneEf-SubSfivef}
The categories $\StoneEf$ and $\SubSfivef$ are equivalent.
\end{proposition}

\begin{proof}
	This follows from the fact that $\StoneER$ and $\SubSfive$ are equivalent allegories (see \cref{thm:StoneER and SubSfive equivalent}), together with the facts that $\StoneEf = \Map(\StoneER)$ and $\SubSfivef = \Map(\SubSfive)$.
\end{proof}

Putting \cref{prop:KH-StoneEf,prop:StoneEf-SubSfivef} together yields: 

\begin{corollary}
	The categories $\KHaus$ and $\SubSfivef$ are equivalent.
\end{corollary}

\begin{definition}
Following \cite[Def.~6.5]{BezhanishviliGabelaiaEtAl2019}, we let $\Gle$ denote the full subcategory of $\StoneEf$ whose objects are Gleason spaces.
We also let $\devF$ denote the full subcategory of $\SubSfivef$ consisting of de Vries algebras. 
\end{definition}

We have $\Gle = \Map(\GleR)$ and $\devF = \Map(\devS)$.

\begin{theorem} \label{t:all-f}
	The categories $\KHaus$, $\StoneEf$, $\Gle$, $\SubSfivef$, and $\devF$ are equivalent.
	\[
	\begin{tikzcd}
		&[-10pt] \StoneEf \arrow[r, shift left= 0.2em]{}  \arrow[ld, shift left= 0.2em]{} &[50pt] \SubSfivef \arrow[l, shift left = 0.2em]{}\\
		\KHaus \arrow[ur, shift left= 0.2em]{} \arrow[dr, shift left= 0.2em]{} & & \\
		& \Gle \arrow[uu, hookrightarrow] \arrow[ul, shift left= 0.2em]{} \arrow[r, shift left= 0.2em]{} & \devF \arrow[uu, hookrightarrow] \arrow[l, shift left= 0.2em]{}
	\end{tikzcd}
	\]
\end{theorem}

\begin{proof}
	By \cref{cor:all-equivalent}, the allegories $\KHausR$, $\StoneER$, $\GleR$, $\SubSfive$ and $\devS$ are equivalent.
	Therefore, the categories $\Map(\KHausR)$, $\Map(\StoneER)$, $\Map(\GleR)$, $\Map(\SubSfive)$, and $\Map(\devS)$ are equivalent. Thus, $\KHaus$, $\StoneEf$, $\Gle$, $\SubSfivef$, and $\devF$ are equivalent. 
\end{proof}

In particular, the equivalence between $\KHaus$ and $\devF$ provides an alternative of de Vries duality. In the rest of the paper we show how to derive de Vries duality from this result. We start by recalling the definition of a de Vries morphism. From now on, following \cref{rem:prec de Vries}, we denote a de Vries proximity on a boolean algebra by $\prec$.

\begin{definition}
	A function $f \colon A \to B$ between de Vries algebras $(A, \prec)$ and $(B, \prec)$ is a {\em de Vries morphism} if it satisfies the following conditions:
	\begin{enumerate}[label = (M\arabic*), ref = M\arabic*]
		\item\label{M1} $f(0) = 0$;
		\item\label{M2} $f(a \land b) = f(a) \land f(b)$;
		\item\label{M3} $a \prec b$ implies $\lnot f(\lnot a) \prec f(b)$;
		\item\label{M4} $f(a) = \bigvee \{f(b) \mid b \prec a\}$.
	\end{enumerate}
	The composition of two de Vries morphisms $f \colon A \to B$ and $g \colon B \to C$ is the de Vries morphism $g * f \colon A \to C$ given by
	\[
		(g * f)(a) = \bigvee \{ gf(b) \mid b \prec a \}
	\]
	for each $a\in A$. Let $\dev$ be the category of de Vries algebras and de Vries morphisms, where identity morphisms are identity functions and composition is defined as above.
\end{definition}

\begin{remark}
Each de Vries morphism $f \colon A \to B$ satisfies that $a \prec b$ implies $f(a) \prec f(b)$ for each $a,b \in A$. This will be used in what follows.
\end{remark}

We recall from the introduction that de Vries duality is induced by the contravariant functor $\RO \colon \KHaus \to \dev$ that associates to each $X \in \KHaus$ the de Vries algebra $(\RO(X),\prec)$ of regular opens of $X$ where $U \prec V$ iff $\cl(U) \subseteq V$. The functor $\RO$ sends each continuous function $f \colon X \to Y$ between compact Hausdorff spaces to the de Vries morphism $f^* \colon \RO(Y) \to \RO(X)$ given by $f^*(V)=\int(\cl(f^{-1}[V]))$ for each $V \in \RO(Y)$.

We show that $\dev$ is dually isomorphic to $\devF$. 
The definition of contravariant functors between $\dev$ and $\devF$ is motivated by the following result.

\begin{proposition} \label{p:translation}
	Let $f \colon X \to Y$ be a continuous function between compact Hausdorff spaces.
	\begin{enumerate}[label=\normalfont(\arabic*), ref = \arabic*]
		\item \label{i:function} For every $V\in \RO(Y)$,
		\[
			\int(\cl(f^{-1}[V])) = \bigvee\{ U \in \RO(X) \mid f[\cl(U)] \subseteq V \},
		\]
		where the join is computed in $\RO(X)$.
		\item \label{i:relation} For every $U \in \RO(X)$ and $V \in \RO(Y)$,
		\[
			f[\cl(U)] \subseteq V \iff \exists V' \in \RO(Y) : \cl(V') \subseteq V \text{ and } U \subseteq \int(\cl(f^{-1}[V'])).
		\]
	\end{enumerate}
\end{proposition}

\begin{proof}
	\eqref{i:function}. 
	Since $f^{-1}[V]$ is open, we have 
	\[
	f^{-1}[V] = \bigcup \{U \in \RO(X) \mid \cl(U) \subseteq f^{-1}[V]\}.
	\]
	Therefore, 
	\[
	\int(\cl(f^{-1}[V])) = \bigvee \{ U \in \RO(X) \mid \cl(U) \subseteq f^{-1}[V] \}\\
	= \bigvee \{ U \in \RO(X) \mid f[\cl(U)] \subseteq V \}.
	\]
	
	\eqref{i:relation}.
	To prove the left-to-right implication, suppose $f[\cl(U)] \subseteq V$.
	Since $f[\cl(U)]$ is closed and $V=\bigcup\{V'\in\RO(Y)\mid\cl(V')\subseteq V\}$, where the union is directed, 
	there is $V' \in \RO(Y)$ such that $f[\cl(U)] \subseteq V' \subseteq \cl(V') \subseteq V$.
	Therefore, $U \subseteq \cl(U) \subseteq f^{-1}[V'] \subseteq \int(\cl(f^{-1}[V']))$.
	
	To prove the right-to-left implication, suppose there is $V' \in \RO(Y)$ such that $\cl(V') \subseteq V$ and $U \subseteq \int(\cl(f^{-1}[V']))$.
	Then 
	\[
	\cl(U) \subseteq \cl(\int(\cl(f^{-1}[V']))) = \cl(f^{-1}[V']) \subseteq \cl(f^{-1}[\cl(V')])= f^{-1}[\cl(V')] \subseteq f^{-1}[V],
	\]
	which implies $f[\cl(U)] \subseteq V$.
\end{proof}

Proposition~\ref{p:translation} suggests the following definition of two contravariant functors providing a dual isomorphism between $\dev$ and $\devF$.

\begin{definition} \label{d:functors}
\mbox{}\begin{enumerate}
\item\label{fS} The contravariant functor from $\devF$ to $\dev$ is the identity on objects and maps a morphism $S \colon A \to B$ in $\devF$ to the function
$f_S \colon B \to A$ given by 
\[
f_S(b) = \bigvee \{ a \in A \mid a \S b \} = \bigvee S^{-1}[b].
\]
\item\label{Sf} The contravariant functor from $\dev$ to $\devF$ is the identity on objects and maps a de Vries morphism $f \colon A \to B$ to the relation $S_f \colon B \to A$ given by
\[
b \S_f a \iff \exists a' \in A : a' \prec a \text{ and } b \leq f(a').
\]
\end{enumerate}
\end{definition}

To show that the above functors are well defined, we need the following lemma.

\begin{lemma} \label{lem:space-dev}
Let $S \colon A \to B$ be a morphism in $\devF$ and let $b_1, b_2 \in B$ be such that $b_1 \prec b_2$. Then
$
f_S(b_1)\S b_2.
$
\end{lemma}

\begin{proof}
By Lemma~\ref{l:explicit}, from $b_1 \prec b_2$ it follows that there is $a_0 \in A$ such that $\lnot a_0 \S \lnot b_1$ and $a_0 \S b_2$.
We show that 
$a \S b_1$ implies $a \leq a_0$.
From $a \S b_1$ and $\neg a_0 \S \neg b_1$ it follows that $(a \wedge \neg a_0) \S (b_1 \wedge \neg b_1) = 0$. By \cref{l:explicit}\eqref{i:total}, $a \wedge \neg a_0=0$, so $a \le a_0$.
Thus, $\bigvee\{a \in A \mid a \S b_1\} \leq a_0$, and hence $\bigvee\{a \in A \mid a \S b_1\} \leq a_0 \S b_2$, which gives $f_S(b_1) \S b_2$.
\end{proof}

\begin{lemma}
The assignment in Definition~\ref{d:functors}(\ref{fS}) is a well-defined contravariant functor from $\devF$ to $\dev$.
\end{lemma}

\begin{proof}
Let $S \colon A \to B$ be a morphism in $\devF$. We show that $f_S$ is a de Vries morphism. That $f_S$ satisfies (\ref{M1}) follows from Lemma~\ref{l:explicit}\eqref{i:total}. The proof of (\ref{M2}) is straightforward and (\ref{M4}) follows from the fact that $S$ is a compatible subordination. We prove (\ref{M3}). Suppose $b_1 \prec b_2$. We must show that $\lnot f_S(\lnot b_1) \prec f_S(b_2)$, i.e. $\lnot \bigvee S^{-1}[\lnot b_1] \prec \bigvee S^{-1}[b_2]$, which is equivalent to 
$\bigwedge \{ \neg c \mid c \in S^{-1}[\lnot b_1] \} \prec \bigvee S^{-1}[b_2]$.
By Lemma~\ref{l:explicit}\eqref{i:deterministic}, there is $a \in A$ such that $\lnot a \mathrel{S} \lnot b_1$ and $a \mathrel{S} b_2$.
Since $S$ is compatible, $a \mathrel{S} b_2$ implies that there is $a' \in A$ such that $a \prec a' \mathrel{S} b_2$.
Because $\neg a \in S^{-1}[\lnot b_1]$ and $a' \in S^{-1}[b_2]$, we have $\bigwedge \{ \neg c \mid c \in S^{-1}[\lnot b_1] \} \leq a \prec a' \leq \bigvee S^{-1}[b_2]$. This proves that $f_S$ is a de Vries morphism.

Let $S_1 \colon A \to B$ and $S_2 \colon B \to C$ be morphisms in $\devF$.
We prove that $f_{S_2 \circ S_1} = f_{S_1} * f_{S_2}$.

\begin{claim} \label{cl:correspondence}
For every $a \in A$ and $c \in C$, 
\[
a \mathrel{(S_2 \circ S_1)} c \iff \exists c' \in C : c' \prec c \mbox{ and } a \leq f_{S_1}f_{S_2}(c').
\]
\end{claim}
\begin{proof}[Proof of claim]
For the left-to-right implication, suppose $a \mathrel{(S_2 \circ S_1)} c$.
Then there is $b \in B$ such that $a \S_1 b \S_2 c$.
From $a \S_1 b$ and the definition of $f_{S_1}$ it follows that $a \leq f_{S_1}(b)$.
Since $S$ is compatible, from $b \S_2 c$ it follows that there is $c' \in C$ such that $b \S_2 c' \prec c$.
Therefore, $b \leq f_{S_2}(c')$, and so
$a \leq f_{S_1}(b) \leq f_{S_1}f_{S_2}(c')$.
		
For the right-to-left implication, suppose that there is $c' \in C$ such that $c' \prec c$ and $a \leq f_{S_1}f_{S_2}(c')$.
By Lemma~\ref{lem:space-dev}, 
$f_{S_2}(c') \S_2 c$. Since $S$ is compatible, there is $b \in B$ such that $f_{S_2}(c') \prec b \S_2 c$.
Applying Lemma~\ref{lem:space-dev} again, $a \leq f_{S_1}f_{S_2}(c') = \bigvee \{ a \in A \mid a \S_1 f_{S_2}(c') \} \S_1 b \S_2 c$, which implies $a \S_1 b \S_2 c$.
Thus, $a \mathrel{(S_2 \circ S_1)} c$.
\end{proof}
For every $c \in C$, we have
\begin{align*}
(f_{S_1}*f_{S_2})(c) & = \bigvee \{ f_{S_1}f_{S_2}(c') \mid c' \prec c \} \\
& = \bigvee\{a \in A \mid \exists c' \in C : c' \prec c \mbox{ and } a  \leq f_{S_1}f_{S_2}(c')\} \\
& = \bigvee\{a \in A \mid a \mathrel{(S_2 \circ S_1)} c\} && \text{(by Claim~\ref{cl:correspondence})} \\
& = f_{S_2 \circ S_1}(c).
\end{align*}
This proves $f_{S_2 \circ S_1} = f_{S_1} * f_{S_2}$.

Let $(A,\prec)$ be a de Vries algebra. Since for every $a \in A$, we have $a = \bigvee\{b \in A \mid b \prec a\}$, the identity on $(A,\prec)$ in $\devF$ is mapped to the identity on $(A,\prec)$ in $\dev$.
\end{proof}

\begin{lemma}
The assignment in Definition~\ref{d:functors}(\ref{Sf}) is a well-defined contravariant functor from $\dev$ to $\devF$.
\end{lemma}

\begin{proof}
Let $f \colon A \to B$ be a de Vries morphism. We show that $S_f \colon B \to A$ is a morphism in $\devF$. It is straightforward to see that $S_f$ is a subordination. We only verify (\ref{S2}). Suppose that $b_1, b_2 \S_f a$. Then there exist $a_1,a_2 \prec a$ such that $b_1 \le f(a_1)$ and $b_2 \le f(a_2)$. Therefore, $(a_1 \vee a_2) \prec a$ and $b_1 \vee b_2 \le f(a_1) \vee f(a_2) \le f(a_1 \vee a_2)$ because $f$ is order-preserving. Thus, $(b_1 \vee b_2) \S_f a$. We next show that $S_f$ is compatible. For this we need to show that $S_f \circ {\prec} = S_f = {\prec} \circ S_f$. To see that $S_f \subseteq S_f \circ {\prec}$, let $b \S_f a$. Then there is $a' \in A$ such that $a' \prec a$ and $b \le f(a')$. By (\ref{S7}), there is $a''\in A$ such that $a' \prec a'' \prec a$. 
Therefore, $b \le f(a') \prec f(a'')$, so $f(a'') \S_f a$, and hence $S_f \subseteq S_f \circ {\prec}$.
The other inclusions are proved similarly. Finally, we show that $S_f$ satisfies the two conditions of \cref{l:explicit}.
Condition \eqref{i:total} is immediate from the definition of $S_f$ and the fact that $f(0)=0$. For \eqref{i:deterministic}, let $a_1 \prec a_2$. By (\ref{S7}), there is $a \in A$ such that $a_1 \prec a \prec a_2$. By (\ref{M3}), $\neg f(\neg a_1) \prec f(a)$. Set $b=f(a)$. It is left to show that $\neg b \S_f \neg a_1$ and $b \S_f a_2$. The latter is obvious because $a \prec a_2$ and $b=f(a)$. 
We prove the former. By (\ref{S7}), there is $c \in A$ such that $a_1 \prec c \prec a$. Then $\neg b = \neg f(a) \prec f(\neg c)$ by (\ref{M3}) and (\ref{S6}), and hence $\neg b \le f(\neg c)$ by (\ref{S5}). Moreover, $\neg c \prec \neg a_1$ by (\ref{S6}). Thus, $\neg b \S_f \neg a_1$. This proves that $S_f$ is a morphism in $\devF$.

Let $f \colon A \to B$ and $g \colon B \to C$ be de Vries morphisms.
We prove that $S_{g * f} = S_f \circ S_g$.
To see that $S_f \circ S_g \subseteq S_{g * f}$, let $a \in A$ and $c \in C$ be such that $c \mathrel{(S_f \circ S_g)} a$.
Then there is $b \in B$ such that $c \S_g b \S_f a$.
Since $c \S_g b$, it follows from the definition of $S_g$ that there is $b' \in B$ such that $b' \prec b$ and $c \leq g(b')$.
Also, since $b \S_f a$, there is $a' \in A$ such that $a' \prec a$ and $b \leq f(a')$.
By (\ref{S7}), there is $a'' \in A$ such that $a' \prec a'' \prec a$. Therefore, $c \leq g(b') \leq g(b) \leq gf(a') \leq (g * f)(a'')$.
This proves $c \S_{g * f} a$.
To see that $S_{g * f} \subseteq S_f \circ S_g$, let $a \in A$ and $c \in C$ be such that $c \S_{g * f} a$.
Then there is $a' \in A$ such that $a' \prec a$ and $c \leq (g * f)(a') \le gf(a')$.
By (\ref{S7}), there is $a'' \in A$ such that $a' \prec a'' \prec a$.
Therefore, $f(a') \prec f(a'')$, and so $c \S_g f(a'') \S_f a$. 
Thus, $c \mathrel{(S_f \circ S_g)} a$.

Let $(A, \prec)$ be a de Vries algebra. If $f$ is the identity on $(A, \prec)$ in $\dev$, then $S_f= {\prec}$, and hence it is the identity on $(A, \prec)$ in $\devF$.
\end{proof}

\begin{theorem} \label{thm:iso-dev-devF}
The contravariant functors described in Definition~\ref{d:functors} establish a dual isomorphism between $\devF$ and $\dev$.
\end{theorem}

\begin{proof}
It is sufficient to show that for each morphism $S \colon A \to B$ in $\devF$ we have $S_{f_S} = S$, and that for each morphism $f \colon A \to B$ in $\dev$ we have $f_{S_f} = f$. 

Let $S \colon A \to B$ be a morphism in $\devF$.
Suppose $a \in A$, $b \in B$, and $a \S_{f_S} b$.
Then there is $b' \in B$ such that $b' \prec b$ and $a \leq f_S(b')$.
By Lemma~\ref{lem:space-dev}, $a \leq f_S(b') \S b$,
so $a \S b$.
This proves $S_{f_S} \subseteq S$.
For the other inclusion, suppose $a \in A$, $b \in B$, and $a \S b$.
Since $S$ is compatible, there is $b' \in B$ such that $a \S b' \prec b$.
We have $a \leq \bigvee\{c \in A \mid c \S b'\} = f_S(b')$.
Therefore, the element $b'$ witnesses that we have $a \S_{f_S} b$.
Thus, $S \subseteq S_{f_S}$, and hence $S_{f_S} = S$.

Let $f \colon A \to B$ be a de Vries morphism. For $a\in A$ we have $f_{S_f}(a) = \bigvee \{ b \in B \mid b \S_f a \}$. Also, $f(a) = \bigvee \{ f(a') \mid a' \prec a \}$ by (\ref{M4}). If $a'\prec a$, then $f(a') \prec f(a)$, so $f(a') \S_f a$. Therefore, $f(a')$ is one of the $b\in B$ such that $b \S_f a$, and hence $f(a) \le f_{S_f}(a)$. On the other hand, if $b \S_f a$, then  there is $a'\in A$ such that $a'\prec a$ and $b \le f(a')$. Therefore, $b\le f(a')\prec f(a)$, and so $f_{S_f}(a)\le f(a)$. Thus, $f_{S_f}(a)=f(a)$, and hence $f_{S_f} = f$.
\end{proof}

\begin{remark}
Combining Theorems~\ref{t:all-f} and~\ref{thm:iso-dev-devF} yields de Vries duality. Consequently, all the categories that appear in Theorem~\ref{t:all-f} are dually equivalent to $\dev$.
\end{remark}

\section*{Acknowledgements}

We thank Sergio Celani for drawing our attention to \cite{Cel18}.
We are also thankful to the referee for suggesting the machinery of order-enriched categories, which has considerably shortened our proofs. \cref{rem:equiv from exact completion,rem:proof exact completion} were inspired by the referee.

Marco Abbadini and Luca Carai were supported by the Italian Ministry of University and Research through the PRIN project n.\ 20173WKCM5 \emph{Theory and applications of resource sensitive logics}.
Luca Carai acknowledges partial support from the Juan de la Cierva-Formaci\'on 2021 programme (FJC2021-046977-I) funded by MCIN/AEI/10.13039/501100011033 and by the European Union ``NextGenerationEU''/PRTR.

\providecommand{\bysame}{\leavevmode\hbox to3em{\hrulefill}\thinspace}
\providecommand{\MR}{\relax\ifhmode\unskip\space\fi MR }
\providecommand{\MRhref}[2]{%
  \href{http://www.ams.org/mathscinet-getitem?mr=#1}{#2}
}
\providecommand{\href}[2]{#2}


\begin{thebibliography}{10}

\bibitem{ABC22b}
M.~Abbadini, G.~Bezhanishvili, and L.~Carai, \emph{Ideal and {M}ac{N}eille
  completions of subordination algebras}, 2022, arXiv:2211.02974.

\bibitem{AbramskyJung1994}
S.~Abramsky and A.~Jung, \emph{Domain theory}, Handbook of Logic in Computer
  Science, vol.~3, Oxford Univ. Press, New York, 1994, pp.~1--168.

\bibitem{Bez10}
G.~Bezhanishvili, \emph{Stone duality and {G}leason covers through de {V}ries
  duality}, Topology Appl. \textbf{157} (2010), no.~6, 1064--1080.

\bibitem{BBSV17}
G.~Bezhanishvili, N.~Bezhanishvili, S.~Sourabh, and Y.~Venema,
  \emph{Irreducible equivalence relations, {G}leason spaces, and de {V}ries
  duality}, Appl. Categ. Structures \textbf{25} (2017), no.~3, 381--401.

\bibitem{BezhanishviliGabelaiaEtAl2019}
G.~{Bezhanishvili}, D.~{Gabelaia}, J.~{Harding}, and M.~{Jibladze},
  \emph{{Compact Hausdorff spaces with relations and Gleason spaces}}, {Appl.
  Categ. Structures} \textbf{27} (2019), no.~6, 663--686.

\bibitem{BRV01}
P.~Blackburn, M.~de~Rijke, and Y.~Venema, \emph{Modal logic}, Cambridge Tracts
  in Theoretical Computer Science, vol.~53, Cambridge University Press,
  Cambridge, 2001.

\bibitem{CV98}
A.~Carboni and E.~M. Vitale, \emph{Regular and exact completions}, J. Pure
  Appl. Algebra \textbf{125} (1998), no.~1-3, 79--116.

\bibitem{CW87}
A.~Carboni and R.~F.~C. Walters, \emph{Cartesian bicategories. {I}}, J. Pure
  Appl. Algebra \textbf{49} (1987), no.~1-2, 11--32.

\bibitem{Celani2001}
S.~A. Celani, \emph{{Quasi-modal algebras}}, {Math. Bohem.} \textbf{126}
  (2001), no.~4, 721--736.

\bibitem{Cel05}
\bysame, \emph{Subdirectly irreducible quasi-modal algebras}, Acta Math. Univ.
  Comenian. (N.S.) \textbf{74} (2005), no.~2, 219--228.

\bibitem{Cel18}
\bysame, \emph{Quasi-semi-homomorphisms and generalized proximity relations
  between {B}oolean algebras}, Miskolc Math. Notes \textbf{19} (2018), no.~1,
  171--189.

\bibitem{deV62}
H.~de~Vries, \emph{Compact spaces and compactifications. {A}n algebraic
  approach}, Ph.D. thesis, University of Amsterdam, 1962.

\bibitem{Eng89}
R.~Engelking, \emph{General topology}, second ed., Sigma Series in Pure
  Mathematics, vol.~6, Heldermann Verlag, Berlin, 1989.

\bibitem{FS90}
P.~J. Freyd and A.~Scedrov, \emph{Categories, allegories}, North-Holland
  Mathematical Library, vol.~39, North-Holland Publishing Co., Amsterdam, 1990.

\bibitem{Hal56}
P.~R. Halmos, \emph{Algebraic logic. {I}. {M}onadic {B}oolean algebras},
  Compositio Math. \textbf{12} (1956), 217--249.

\bibitem{HV19}
C.~Heunen and J.~Vicary, \emph{Categories for quantum theory}, Oxford Graduate
  Texts in Mathematics, vol.~28, Oxford University Press, Oxford, 2019.

\bibitem{HST14}
D.~Hofmann, G.~J. Seal, and W.~Tholen (eds.), \emph{Monoidal topology. {A}
  categorical approach to order, metric, and topology}, Encycl. Math. Appl.,
  vol. 153, Cambridge: Cambridge University Press, 2014.

\bibitem{Hoofman1993}
R.~Hoofman, \emph{Continuous information systems}, Inform. and Comput.
  \textbf{105} (1993), no.~1, 42--71.

\bibitem{Isb72}
J.~Isbell, \emph{Atomless parts of spaces}, Math. Scand. \textbf{31} (1972),
  5--32.

\bibitem{Joh82}
P.~T. Johnstone, \emph{Stone spaces}, Cambridge Studies in Advanced
  Mathematics, vol.~3, Cambridge University Press, Cambridge, 1982.

\bibitem{Joh02}
\bysame, \emph{Sketches of an elephant: a topos theory compendium. {V}ol. 1},
  Oxford Logic Guides, vol.~43, The Clarendon Press, Oxford University Press,
  New York, 2002.

\bibitem{JungKegelmannEtAl1999}
A.~Jung, M.~Kegelmann, and M.~A. Moshier, \emph{Multilingual sequent calculus
  and coherent spaces}, Fund. Inform. \textbf{37} (1999), no.~4, 369--412.

\bibitem{JungKegelmannEtAl2001}
\bysame, \emph{Stably compact spaces and closed relations}, Electron. Notes
  Theor. Comput. Sci. \textbf{45} (2001), 209--231.

\bibitem{KurzMoshierEtAl2019}
A.~Jung, A.~Kurz, and M.~A. Moshier, \emph{Stone duality for relations}, 2019,
  arXiv:1912.08418.

\bibitem{JungSuenderhauf1996}
A.~Jung and P.~S\"{u}nderhauf, \emph{On the duality of compact vs. open},
  Papers on General Topology and Applications ({G}orham, {ME}, 1995), Ann. New
  York Acad. Sci., vol. 806, New York, 1996, pp.~214--230.

\bibitem{Kegelmann2002}
M.~Kegelmann, \emph{Continuous domains in logical form}, Electron. Notes Theor.
  Comput. Sci., vol.~49, Elsevier Science B.V., Amsterdam, 2002.

\bibitem{Lambek1999}
J.~Lambek, \emph{Diagram chasing in ordered categories with involution}, J.
  Pure Appl. Algebra \textbf{143} (1999), no.~1-3, 293--307.

\bibitem{LarsenWinskel1984}
K.~G. Larsen and G.~Winskel, \emph{Using information systems to solve recursive
  domain equations effectively}, Semantics of data types ({V}albonne, 1984),
  Lecture Notes in Comput. Sci., vol. 173, Springer, Berlin, 1984,
  pp.~109--129.

\bibitem{Mac98}
S.~Mac~Lane, \emph{Categories for the working mathematician}, second ed.,
  Graduate Texts in Mathematics, vol.~5, Springer-Verlag, New York, 1998.

\bibitem{MR20}
V.~Marra and L.~Reggio, \emph{A characterisation of the category of compact
  {H}ausdorff spaces}, Theory Appl. Categ. \textbf{35} (2020), Paper No. 51,
  1871--1906.

\bibitem{Moshier2004}
M.~A. Moshier, \emph{On the relationship between compact regularity and
  {Gentzen}'s cut rule}, Theor. Comput. Sci. \textbf{316} (2004), no.~1-3,
  113--136.

\bibitem{PW88}
J.~R. Porter and R.~G. Woods, \emph{Extensions and absolutes of {H}ausdorff
  spaces}, Springer-Verlag, New York, 1988.

\bibitem{Scott1982}
D.~S. Scott, \emph{Domains for denotational semantics}, Automata, languages and
  programming ({A}arhus, 1982), Lecture Notes in Comput. Sci., vol. 140,
  Springer, Berlin-New York, 1982, pp.~577--613.

\bibitem{TsalenkoGisinEtAl1984}
M.~Sh. Tsalenko, V.~B. Gisin, and D.~A. Rajkov, \emph{Ordered categories with
  involution}, Diss. Math. \textbf{227} (1984).

\bibitem{Vic11}
J.~Vicary, \emph{Completeness of {$\dagger$}-categories and the complex
  numbers}, J. Math. Phys. \textbf{52} (2011), no.~8, 082104, 31.

\bibitem{Vickers1993}
S.~Vickers, \emph{Information systems for continuous posets}, Theor. Comput.
  Sci. \textbf{114} (1993), no.~2, 201--229.

\end{thebibliography}
\end{document}